\documentclass[3p]{elsarticle}

\usepackage{lineno,hyperref}
\usepackage[english]{babel}
 \usepackage{amssymb}
\usepackage{multicol}
\usepackage{amsfonts}
\usepackage{algpseudocode} 
\usepackage{algorithm} 
\usepackage{graphicx}
\usepackage{caption}
\usepackage{bm}
\usepackage{bbm} 
\usepackage{amsthm} 
\usepackage{amsmath}
\usepackage{color}
\usepackage{picture}
\usepackage{tikz}
\usepackage{subfig}
\usepackage{adjustbox} 

\newcommand{\vect}[1]{\boldsymbol{#1}}

\renewcommand{\d}{\ \mathrm{d}}
  
\newcommand{\z}{\phantom{0}}
\renewcommand{\d}{\mathrm{d}}

 \usepackage{cases}

\newtheorem{ass}{Assumption}
\newtheorem{thm}{Theorem} 
\newtheorem{prop}{Proposition}
\newtheorem{rmk}{Remark}

\numberwithin{equation}{section}

\renewcommand{\d}{\mathrm{d}}
\newcommand{\dt}{\mathrm{dt}}

\def\R1{\color{blue}}

\definecolor{ao(english)}{rgb}{0.1, 0.5, 0.1}

\renewcommand{\S}{\mathcal{S}}

\newcommand{\vecx}{\boldsymbol{x}}

\begin{document}

\begin{frontmatter}

\title{  An efficient solver for space-time isogeometric Galerkin methods for parabolic problems  }
  \author[mat]{Gabriele Loli}\ead{gabriele.loli01@universitadipavia.it}
 \author[mat]{Monica Montardini\corref{cor1}} \ead{monica.montardini01@universitadipavia.it}
\author[mat,imati]{Giancarlo Sangalli} \ead{giancarlo.sangalli@unipv.it}
\author[imati]{Mattia Tani} \ead{mattia.tani@imati.cnr.it}

\address[mat]{Dipartimento di Matematica ``F. Casorati", Universit\`{a} di Pavia, Via A. Ferrata, 5, 27100 Pavia, Italy.}

\address[imati]{Istituto di Matematica Applicata e Tecnologie Informatiche, ``E. Magenes" del CNR, Via A. Ferrata, 1, 27100 Pavia, Italy.}
 \cortext[cor1]{Corresponding author}
\begin{abstract}
  In this work we focus on  the  preconditioning of a  Galerkin
  space-time isogeometric discretization of  the heat equation.  
     Exploiting the tensor
   product structure of the basis functions  in the parametric domain,  we propose a
   preconditioner that is the sum of Kronecker products of matrices and that can be efficiently applied thanks to an extension of the
   classical Fast Diagonalization method. 
The preconditioner is robust w.r.t. the polynomial degree   of the spline space   and the time
required for the application  is almost proportional to the number of
degrees-of-freedom, for a serial execution.  By incorporating some
information on the geometry parametrization and on  the equation coefficients, we
keep high efficiency with non-trivial domains and variable thermal conductivity and heat capacity coefficients.
\end{abstract}
\begin{keyword}
  Isogeometric Analysis \sep splines \sep heat equation \sep space-time
 Galerkin  formulation \sep Fast Diagonalization.
\end{keyword}
\end{frontmatter}



%

\section{Introduction}

{Isogeometric Analysis (IgA), introduced in the seminal
paper \cite{Hughes2005} (see also the book \cite{Cottrell2009}), is an evolution of the classical finite
element methods. IgA  uses spline  functions, or their
generalizations, both to  represent the computational domain and to approximate the solution of the partial
differential equation that models the problem of interest. This is
meant to simplify the interoperability between computer aided design
and numerical simulations.
IgA also benefits  from the approximation
properties of splines,  whose 
high-continuity   yields  higher accuracy when compared to  
$C^0$ piecewise polynomials, see e.g.,
\cite{Evans_Bazilevs_Babuska_Hughes,bressan2018approximation,Sangalli2018}.}

In this paper we focus on  the  heat equation and on its space-time
Galerkin isogeometric discretization.
Space-time finite element methods originated  in the papers
\cite{fried1969finite, bruch1974transient,oden1969general2},  
 where standard finite elements are ascribed an extra dimension for the time  and, typically,  adopt a  discontinuous approximation in time, since this produces a time
marching algorithm with a traditional step-by-step format (see
e.g. \cite{shakib1991new}).

 One of the first work concerning space-time isogeometric discretization is \cite{LANGER2016342}, in which a stabilized variational formulation produces a discrete bilinear form that is elliptic with respect to a discrete energy norm. The resulting linear system is then solved through a standard parallel AMG preconditioned GMRES solver.
Other papers in literature propose isogeometric space-time  Galerkin 
methods, favoring a  step-by-step structure in time. In
\cite{Langer2017}, the same variational formulation of
\cite{LANGER2016342} is used in combination with a space-time domain
decomposition into space-time slabs that are sequentially coupled  in
time by a  stabilized discontinuous Galerkin method.  In
  \cite{takizawa2014spacetime}, two different methods, called ST-C-SPT
  and ST-C-DCT, are outlined. The first one, analysed in
  \cite{UEDA2019266},  is a way to project a previously computed
  solution, possibly discontinuous, into isogeometric spaces in order
  to get a more regular solution and to save memory for its
  storage. In the ST-C-DCT method, the solution with continuous
  temporal representation is computed sequentially  from the
  space-time variational formulation associated with each slab.

Related multigrid solvers have   been proposed in
\cite{Gander2016,hofer2019parallel} and low-rank approximations in \cite{mantzaflaris2019low}. In \cite{bonilla2019maximum}  the
authors consider  $C^0$ coupling between the space-time slabs  with a
suitable  stabilized formulation  that also  yields a sequential scheme.    
 Finally, the interest in  space-time isogeometric
analysis for complex real-world simulations is attested by the recent papers
\cite{takizawa2017turbocharger,takizawa2016ram,takizawa2018heart},
where, again,  a sequential (discontinuous) approximation in time is adopted.  

The novelty of our  work is that we deal with   smooth approximation in both  space and time. 
This started in a  previous work,  \cite{Montardini2018space},  based
on  a $L^2$ least-squares formulation. The reason of this choice is that the problem becomes
elliptic and a preconditioner for the linear system can be easily designed as in  \cite{Sangalli2016}.
Indeed, when adopting smooth approximation in space and in time, the
major issue is its computational cost and the  
key ingredient  is an efficient solver for  the linear system, which is 
global in time.
In the present work, instead, we focus  on the plain  Galerkin space-time  formulation,
whose well-posedness has been studied, for finite element discretizations and  for the
heat equation, in the recent papers \cite{Steinbach2015} and 
\cite{Stevenson2019}. { For a Galerkin formulation, and assuming  
that the spatial domain does not change with time,  the linear system
has   the structure
\begin{equation}\label{eq:system-matrix}
	\gamma   \mathbf{W}_t\otimes  \mathbf{M}_s +   \nu \mathbf{M}_t\otimes  \mathbf{K}_s, 
\end{equation}
where $ \mathbf{W}_t$ is given by  the discretization of the
time derivative, $\mathbf{K}_s $ is given by the discretization of the
Laplacian in the spatial variables, $\mathbf{M}_t $ and $\mathbf{M}_s $
are ``mass matrices'' in time and space, respectively,  and $\gamma,\nu>0$ are constants of the problem.
Adopting an iterative solver, we do not need to form the matrix
\eqref{eq:system-matrix} (observe that the cost of formation of the matrices in
\eqref{eq:system-matrix} is comparable to the cost of  forming a
steady-state  diffusion matrix) but there is the need of an  efficient
preconditioning strategy. The main contribution of this paper is the
construction of  a
preconditioner for \eqref{eq:system-matrix} generalizing the classical
Fast Diagonalization (FD) method \cite{Lynch1964}. 
Indeed,  the  FD method cannot be directly applied to \eqref{eq:system-matrix}, as this would require to compute the eigendecomposition of the pencil $( \mathbf{W}_t,\mathbf{M}_t)$ which is numerically unstable. We circumvent this difficulty by introducing an ad-hoc factorization of the time matrices which allows to design a solver conceptually similar to the FD method. The computational cost of the setup of the resulting preconditioner is $O(N_{dof})$ FLOating-Point operations (FLOPs)  while its  application  is  $O(N_{dof}^{1 + 1/d})$ FLOPs, where $d$ is the number of spatial dimensions and  $N_{dof}$ denotes the total number of degrees-of-freedom  (assuming, for simplicity,  to have   the same number of  degrees-of-freedom
in time and in each spatial direction). Our numerical benchmarks show
that the computing time (serial and single-core
execution)  is close to optimality, that is, proportional to
$N_{dof}$. The preconditioner is also robust with
respect to the polynomial degree. Furthermore, our approach is 
optimal in terms of memory requirement: denoting  by
$N_s$  the total number of degrees-of-freedom in space,  the storage
cost is $O(p^d N_s + N_{dof})$. 
We also remark that global space-time methods in principle  facilitate the full 
parallelization of the solver,  see \cite{dorao2007parallel,Gander2015,kvarving2011fast}.

The outline of the paper is as follows. In Section 2 we present the basics of B-splines based IgA and the main properties of the Kronecker product operation. The model problem and its isogeometric discretization are introduced in Section 3, while in Section 4 we define the preconditioner and we discuss its application. We present the numerical results assessing the performance of the proposed preconditioner in Section 5. Finally, in the last section we draw some conclusions   and  we  highlight some future research directions.}

\section{Preliminaries}

\subsection{B-Splines}

Given $m$ and $p$ two positive integers, a 
 knot vector in $[0,1]$ is   a sequence of non-decreasing points
 $\Xi:=\left\{ 0=\xi_1 \leq \dots \leq \xi_{m+p+1}=1\right\}$. 
 We consider open knot vectors, i.e.  we set $\xi_1=\dots=\xi_{p+1}=0$
 and $\xi_{m}=\dots=\xi_{m+p+1}=1$. 
 Then, according to   Cox-De Boor recursion formulas (see  \cite{DeBoor2001}), univariate B-splines $\widehat{b}_{i,p}:(0,1)\rightarrow \mathbb{R}$ are  piecewise polynomials defined  for $i=1,\dots,m$ as \\
\indent for $p=0$:
\begin{align*} 
\widehat{b}_{i,0}(\eta) = \begin{cases}1 &  { \textrm{if }} \xi_{i}\leq \eta<\xi_{i+1},\\
0 & \textrm{otherwise,}
\end{cases}
\end{align*}
\indent for $p\geq 1$: 
\begin{align*}
 \widehat{b}_{i,p}(\eta)=\! \begin{cases}\dfrac{\eta-\xi_{i}}{\xi_{i+p}-\xi_{i}}\widehat{b}_{  i,p-1}(\eta)   +\dfrac{\xi_{i+p+1}-\eta}{\xi_{i+p+1}-\xi_{i+1}}\widehat{b}_{  i+1,p-1}(\eta)   & { \textrm{if }} \xi_{i}\leq  \eta<\xi_{i+p+1}, \\[8pt]
0  & \textrm{otherwise,}
\end{cases}
\end{align*}
 where we adopt the convention $0/0=0$. 
The univariate spline space is defined as
\begin{equation*}
\widehat{\S}_h^p : = \mathrm{span}\{\widehat{b}_{i,p}\}_{i = 1}^m,
\end{equation*}
where $h$ denotes the mesh-size, i.e. $ h:=\max\{ |\xi_{i+1}-\xi_i| \ | \ i=1,\dots,m+p \}$.
The interior knot multiplicity
  influences the smoothness of the B-splines at
the knots (see \cite{DeBoor2001}). For more details on B-splines properties  and their use in IgA we refer to  \cite{Cottrell2009}.

Multivariate B-splines are defined as tensor product of univariate
B-splines.  
We consider  functions that depend on $d$ spatial variables  and the time variable. Therefore, given positive integers $m_l, p_l$ for $l=1,\dots,d$ and $m_t,p_t$, we  introduce $d+1 $
univariate knot vectors $\Xi_l:=\left\{ \xi_{l,1} \leq \dots \leq
  \xi_{l,m_l+p_l+1}\right\}$  for $l=1,\ldots, d$  and   $\Xi_t:=\left\{
  \xi_{t,1} \leq \dots \leq \xi_{t,m_t+p_t+1}\right\}$.
  Let $h_l$ be the mesh-size associated to the knot vector $\Xi_l$ for $l=1,\dots,d$, let $h_s:=\max\{h_l\ | \ l=1,\dots, d\}$ be the maximal mesh-size in all spatial knot vectors and let $h_t$ be the mesh-size of the time knot vector. 
  Let also $\boldsymbol{p}$ be the vector that contains the degree indexes, i.e. $\boldsymbol{p} :=(\boldsymbol{p}_s,
p_t)$, where $\boldsymbol{p}_s:= 
 (p_1,\dots,p_d )$. For simplicity, we assume to have the same polynomial degree in all spatial directions, i.e., with abuse of notations, we set  $p_1=\dots=p_d=:p_s$, but the general case is similar.
 
We assume that the following quasi-uniformity of the knot vectors
holds.

\begin{ass}
\label{ass:quasi-uniformity}
 There exists  $0<\alpha\leq 1 $, independent
of $h_s$ and $h_t$, such that each non-empty knot span $( \xi_{l,i} ,
\xi_{l,i+1})$ of $\Xi_l$ fulfils $ \alpha h_s  \leq \xi_{l,i+1} - \xi_{l,i} \leq
h_s$  for $ l =1,\dots,  d$   and  each non-empty knot-span $( \xi_{t,i} ,
\xi_{t,i+1})$ of $\Xi_t$ fulfils $ \alpha h_t  \leq \xi_{t,i+1} - \xi_{t,i} \leq
h_t$.
\end{ass}

The multivariate B-splines are defined as
\begin{equation*} 
\widehat{B}_{ \vect{i},\vect{p}}(\vect{\eta},\tau) : =
\widehat{B}_{\vect{i_s}, \vect{p}_s}(\vect{\eta}) \widehat{b}_{i_t,p_t}(\tau),
\end{equation*}
 where 
 \begin{equation}
   \label{eq:tens-prod-space-parametric}
   \widehat{B}_{\vect{i_s},\vect{p}_s}(\vect{\eta}):=\widehat{b}_{i_1,p_s}(\eta_1) \ldots \widehat{b}_{i_d,p_s}(\eta_d),
 \end{equation}
  $\vect{i_s}:=(i_1,\dots,i_d)$, $\vect{i}:=(\vect{i_s}, i_t)$  and  $\vect{\eta} = (\eta_1, \ldots, \eta_d)$.  
The  corresponding spline space  is defined as
\begin{equation*}
\widehat{\S} ^{\vect{p}}_{ {h}  }  := \mathrm{span}\left\{\widehat{B}_{\vect{i}, \vect{p}} \ \middle| \ i_l = 1,\dots, m_l \text{ for } l=1,\dots,d; i_t=1,\dots,m_t \right\},
\end{equation*} 
  where $h:=\max\{h_s,h_t\}$. 
We have that
$\widehat{\S} ^{\vect{p}}_{ {h}}  =\widehat{\S} ^{ \vect{p}_s}_{ {h}_s} \otimes
\widehat{\S} ^{p_t}_{h_t}, $ where   \[\widehat{\S} ^{\vect{p}_s}_{h_s}
:= \mathrm{span}\left\{
  \widehat{B}_{\vect{i_s},\vect{p}_s} \ \middle|  \ i_l =
  1,\dots, m_l; l=1,\dots,d  \right\}\] is the space of
tensor-product splines on    $\widehat{\Omega}:=(0,1)^d$.  

\begin{ass} 
\label{ass:knot_mult}
We assume that $p_t, p_s\geq 1$ and that $\widehat{\S} ^{\vect{p}_s}_{h_s}
\subset C^0(\widehat{\Omega}  )$ and  $\widehat{\S} ^{{p}_t}_{h_t}
\subset C^0\left((0,1)\right)$ . 
\end{ass}
 
\subsection{Isogeometric spaces}
\label{sec:iso_space}
The space-time computational domain that we consider is $\Omega\times (0,T)$, where $\Omega\subset\mathbb{R}^d$ and $T>0$ is the final time. We make the following assumption.
 
\begin{ass}
\label{ass: regular-single-patch-domain}
 We assume  that  $\Omega$  is parametrized by  $\vect{F}: \widehat{\Omega} \rightarrow {\Omega}
$, with  $\vect{F}\in  {\left[\widehat{\mathcal{S}}^{\vect{p}_s}_{{h}_s}\right]^d} $. Moreover, we assume that
  $\vect{F}^{-1}$ has piecewise
  bounded derivatives of any order.
\end{ass}

 We define  $\vecx=(x_1,\dots,x_d):=  \vect{F}(\vect{\eta})$ and
 $t:=T\tau$. Then the space-time domain  is given by  the parametrization
 $\vect{G}:\widehat{\Omega}\times(0,1)\rightarrow
 \Omega\times(0,T)$, such that $ \vect{G}(\vect{\eta}, \tau):=(\vect{F}(\vect{\eta}), T\tau )=(\vecx,t).$
 
 We introduce the spline space with initial and boundary
 conditions, in parametric coordinates, as  
\begin{equation*}
\widehat{\mathcal{X}}_{h}:=\left\{ \widehat{v}_h\in \widehat{\mathcal{S}}^{\vect{p}}_h \ \middle| \ \widehat{v}_h = 0 \text{ on } \partial\widehat{\Omega}\times (0,1) \text{ and } \widehat{v}_h = 0 \text{ on } \widehat{\Omega}\times\{0\} \right\}.
\end{equation*}
  We also have that  $
\widehat{\mathcal{X}}_{h} =   \widehat{\mathcal{X}}_{s,h_s}     \otimes  \widehat{\mathcal{X}}_{t,h_t}  $,  where 
 \begin{subequations}
 \begin{align*}
  \widehat{\mathcal{X}}_{s,h_s}   & := \left\{ \widehat{w}_h\in \widehat{\mathcal{S}}^{\vect{p}_s}_{h_s}    \ \middle| \ \widehat{w}_h = 0 \text{ on } \partial\widehat{\Omega}  \right\}\  \\
  &  \ = \ \text{span}\left\{ \widehat{b}_{i_1,p_s}\dots\widehat{b}_{i_d,p_s} \ \middle| \ i_l = 2,\dots , m_l-1; \ l=1,\dots,d\ \right\},\\ 
   \widehat{\mathcal{X}}_{t,h_t} & := \left\{ \widehat{w}_h\in \widehat{\mathcal{S}}^{ p_t}_{h_t} \ \middle|  \ \widehat{w}_h( 0)=0 \right\}  \   = \ \text{span}\left\{ \widehat{b}_{i_t,p_t} \ \middle| \ i_t = 2,\dots , m_t\ \right\}.
 \end{align*}
 \end{subequations}
By introducing a colexicographical reordering of  the basis functions, we can  write  
\begin{align*}
  \widehat{\mathcal{X}}_{s,h_s} &   = \ \text{span}\left\{ \widehat{b}_{i_1,p_s}\dots\widehat{b}_{i_d,p_s} \ \middle| \ i_l = 1,\dots , n_{s,l}; \ l=1,\dots,d\ \right\}\\
  & \ =\text{span}\left\{ \widehat{B}_{i, \vect{p}_s} \ \middle|\ i =1,\dots , N_s   \ \right\},\\
\  \widehat{\mathcal{X}}_{t,h_t}  &  = \ \text{span}\left\{ \widehat{b}_{i,p_t} \ \middle| \ i = 1,\dots , n_t\ \right\},
\end{align*}
  and then
\begin{equation}
\widehat{\mathcal{X}}_{h}=\text{span}
\left\{ \widehat{B}_{{i}, \vect{p}} \ \middle|\ i=1,\dots,N_{dof} \right\},
\label{eq:all_basis}
\end{equation}
where we defined $n_{s,l}:= m_l-2 $ for $l=1,\dots,d$, $N_s:=\prod_{l=1}^dn_{s,l}$, $n_t:=m_t-1$ and $ N_{dof}:=N_s n_t$. 

Finally, the isogeometric space we consider is the isoparametric push-forward of \eqref{eq:all_basis} through the geometric map $\vect{G}$, i.e.
\begin{equation}
\mathcal{X}_{h} := \text{span}\left\{  B_{i, \vect{p}}:=\widehat{B}_{i, \vect{p}}\circ \vect{G}^{-1} \ \middle| \ i=1,\dots , N_{dof}   \right\}.
\label{eq:disc_space}
\end{equation}
We also have that   
 $\mathcal{X}_{h}=\mathcal{X}_{s,h_s}   \otimes \mathcal{X}_{t,h_t} $,
   where 
\[ 
 \mathcal{X}_{s,h_s}    :=\text{span}\left\{ {B}_{i, \vect{p}_s}:= \widehat{B}_{i, \vect{p}_s}\circ \vect{F}^{-1} \ \middle| \ i=1,\dots,N_s \right\}\]
 and \[ \mathcal{X}_{t,h_t}   :=\text{span}\left\{  {b}_{i,p_t}:= \widehat{b}_{i,p_t}( \cdot /T) \ \middle| \ i=1,\dots,n_t \right\}.
\]
 \subsection{Kronecker product}
 
The Kronecker product of two matrices $\mathbf{C}\in\mathbb{C}^{n_1\times n_2}$ and $\mathbf{D}\in\mathbb{C}^{n_3\times n_4}$ is defined as
\[
\mathbf{C}\otimes \mathbf{D}:=\begin{bmatrix}
[\mathbf{C}]_{1,1}\mathbf{D}  & \dots& [\mathbf{C}]_{1,n_2}\mathbf{D}\\
\vdots& \ddots &\vdots\\
[\mathbf{C}]_{n_1, 1}\mathbf{D}& \dots & [\mathbf{C}]_{n_1, n_2}\mathbf{D}
\end{bmatrix}\in \mathbb{C}^{n_1n_3\times n_2 n_4},
\]
where   $[\mathbf{C}]_{i,j}$ denotes the $ij$-th entry of the matrix $\mathbf{C}$. 
For extensions and   properties of the Kronecker product   we refer to \cite{Kolda2009}.
In particular, when   a  matrix   has a Kronecker product structure, the matrix-vector product   can be efficiently computed. For this
purpose, for $m=1,\dots,d+1$  we  introduce the $m$-mode product   of a tensor $\mathfrak{X}\in\mathbb{C}^{n_1\times\dots\times n_{d+1}}$ with a matrix $\mathbf{J}\in\mathbb{C}^{\ell\times n_m}$,   that we denote by  $\mathfrak{X}\times_m \mathbf{J}$.  This is 
  a tensor of size $n_1\times\dots\times n_{m-1}\times \ell \times n_{m+1}\times \dots n_{d+1}$,   whose elements are  defined as
\[\left[ \mathfrak{X}\times_m \mathbf{J} \right]_{i_1, \dots, i_{d+1}} := \sum_{j=1}^{n_m} [\mathfrak{X}]_{i_1,,\dots, i_{m-1},j,i_{m+1},\dots,i_{d+1}}[\mathbf{J}]_{i_m,j }.\]
Then, given $\mathbf{J}_i\in\mathbb{C}^{\ell_i\times n_i}$ for $i=1,\dots, d+1$, it holds
\begin{equation}
\label{eq:kron_vec_multi}
\left(\mathbf{J}_{d+1}\otimes\dots\otimes \mathbf{J}_1\right)\text{vec}\left(\mathfrak{X}\right)=\text{vec}\left(\mathfrak{X}\times_1 \mathbf{J}_1\times_2 \dots \times_{d+1}\mathbf{J}_{d+1} \right),
\end{equation} 
 where the vectorization operator ``vec" applied to a tensor stacks its entries  into a column vector as
 \[ [\text{vec}(\mathfrak{X})]_{j}=[\mathfrak{X}]_{i_1,\dots,i_{d+1}} \text{ for }   i_l=1,\dots,n_{l} \text{ and for } l=1,\dots,d+1,   \] 
where  $j:=i_1+\sum_{k=2}^{d+1}\left[(i_k-1)\Pi_{l=1}^{k-1}n_l\right]$.
 
\section{The model problem}
\label{sec:problem}

\subsection{Space-time variational formulation} 
Our model problem is the heat equation with homogeneous boundary and initial conditions:   we look for a solution $u$ such that 
\begin{equation}
\label{eq:problem}
	\left\{
	\begin{array}{rcllrcl}
		 \gamma\partial_t u -  \nabla\cdot(\nu \nabla u)  & = & f & \mbox{in } &\Omega \!\!\!\! &\times & \!\!\!\! (0, T),\\[1pt]
		 u  & = & 0 & \mbox{on } &\partial\Omega \!\!\!\! &\times& \!\!\!\! [0, T],\\[1pt]
		 u & = & 0 & \mbox{in } &\Omega\!\!\!\!  &\times &\!\!\!\! \lbrace 0 \rbrace,
	\end{array}
	\right.
\end{equation}
where $\Omega \subset \mathbb{R}^d$, $T$ is the final time,
$\gamma>0$ is the heat capacity constant and   $\nu > 0$ is the
thermal conductivity constant.   
We assume that  $f\in L^2( 0,T; H^{-1}(\Omega))$ and we introduce the Hilbert spaces
\[
 \mathcal{X}:=\left\{v\in L^2(0,T;H_0^1(\Omega))\cap
  H^1(0,T;H^{-1}(\Omega)) \mid v(\vecx,0)=0 \right\},\]
  \[  \mathcal{Y}:= L^2(0,T;H_0^1(\Omega)),
\]
endowed with the following norms  
\begin{equation*}
  \|v\|_{ \mathcal{X}}^2:= \frac{\gamma^2}{ {\nu}}  
  \|\partial_tv\|^2_{L^2(0,T;H^{-1}(\Omega))}+ \nu \|v\|^2_{L^2(0,T;H^1_0(\Omega))}  \ \text{ and }\ \|v\|_{ \mathcal{Y}}^2:= \nu  \|v\|^2_{L^2(0,T;H^1_0(\Omega))},
\end{equation*}
respectively.   Then, the variational formulation of \eqref{eq:problem} reads:  
\begin{equation}
\label{eq:var_for}
\text{Find } u \in \mathcal{X} \text{ such that }  \mathcal{A}( {u},v) =
\mathcal{F} (v)  \quad \, \forall v \in  \mathcal{Y},
\end{equation}  
where the bilinear form $\mathcal{A}(\cdot,\cdot)$ and the linear form $\mathcal{F}(\cdot)$ are defined $ \forall w\in\mathcal{X} \text{ and } \forall v\in \mathcal{Y}$ as
\[
 \mathcal{A}(w,v) := \int_{0}^T\int_{\Omega}  \left(  \gamma\partial_t {w} \,  v +  \nu\nabla w\cdot \nabla v \right)\,\d\Omega\,   \dt  
\quad \text{and} \quad
\mathcal{F}(v)  := \int_0^T\int_{\Omega}f\,   v \,\d\Omega\,\dt.
\] 
The well-posedness  of the variational formulation above is a
classical result, see for example \cite{Steinbach2015}.    
 
The previous setting can be generalized  to    non-homogeneous initial and boundary conditions.
For example, suppose that in \eqref{eq:problem} we have the initial condition $u=u_0$ in $\Omega\times\{0\}$ with  $u_0\in L^2(\Omega).$    Then, we consider a lifting  $\underline{u}_0$ of $u_0$ such that $\underline{u}_0\in L^2(0,T;H_0^1(\Omega))\cap
H^1(0,T;H^{-1}(\Omega))$, see e.g. \cite{Evans2010book}. Finally, we split the solution $u$ as $u=\underline{u}+\underline{u}_0$, where $\underline{u}\in\mathcal{X}$ is the solution of the following heat equation with homogeneous initial and boundary conditions:
 \begin{equation*} 
	\left\{
	\begin{array}{rcllrcl}
		 \gamma\partial_t \underline{u} -  \nabla\cdot(\nu \nabla \underline{u})  & = & \underline{f} & \mbox{in } &\Omega \!\!\!\! &\times & \!\!\!\! (0, T),\\[1pt]
		 \underline{u}  & = & 0 & \mbox{on } &\partial\Omega \!\!\!\! &\times& \!\!\!\! [0, T],\\[1pt]
		 \underline{u} & = & 0 & \mbox{in } &\Omega\!\!\!\!  &\times &\!\!\!\! \lbrace 0 \rbrace,
	\end{array}
	\right.
	 \end{equation*}
	 where $\underline{f}:=f-\gamma\partial_t \underline{u}_0 +  \nabla\cdot(\nu \nabla \underline{u}_0)$.

\subsection{Space-time Galerkin method}
Let  $\mathcal{X}_h\subset \mathcal{X}$ be  the isogeometric space
defined in  \eqref{eq:disc_space}.  We consider the following Galerkin method for \eqref{eq:var_for}: 
\begin{equation}
\label{eq:discrete-system}
 \text{Find }   u_h\in \mathcal{X}_h \text{ such that } \mathcal{A}(u_h, v_h) = \mathcal{F}(v_h) \quad \, \forall v_h\in  \mathcal{X}_h.
\end{equation} 
Following \cite{Steinbach2015}, let $N_h: L^2(0,T;H^{-1}(\Omega))\rightarrow \mathcal{X}_h$ be the discrete Newton potential operator, defined as follows: given $\phi\in  L^2(0,T;H^{-1}(\Omega))$ then $N_h\phi\in\mathcal{X}_h$ fulfills
\begin{equation*}
 \int_{0}^T\int_{\Omega}{ \nu}\nabla (N_h\phi)\cdot\nabla v_h \,\d\Omega\, \dt=\gamma\int_{0}^T\int_{\Omega}\phi\ v_h\,\d\Omega\, \dt\quad \forall v_h\in  \mathcal{X}_h.
\end{equation*}
Thus, we define the  norm in $\mathcal{X}_h$ as 
\begin{equation*}
 \|w\|^2_{\mathcal{X}_h}:=  \nu\|N_h(\partial_t w)\|^2_{L^2(0,T;H^1_0(\Omega))} + { \nu }\|w\|^2_{ L^2(0,T;H^1_0(\Omega))}.
\end{equation*}

The stability and the  well-posedness of    formulation \eqref{eq:discrete-system} are
guaranteed by  \cite[Equation (2.7)]{Steinbach2015} and by   a   straightforward   extension to IgA of  \cite[Theorem 3.1]{Steinbach2015} and \cite[Theorem 3.2]{Steinbach2015}. We summarize these results in the following Proposition \ref{prop:discrete_infsup} and Theorem \ref{thm:quasi-optimality}.

\begin{prop}
 \label{prop:discrete_infsup}
It holds
\begin{equation*} 
  \mathcal{A}(w ,v )\leq \sqrt{2} \|w \|_{\mathcal{X}}
\|v \|_{\mathcal{Y}} \quad \forall w  \in  \mathcal{X} \text{ and } \forall v \in  \mathcal{Y},
\end{equation*}
and
\begin{equation*} 
\|w_h\|_{ \mathcal{X}_h}\leq {  2\sqrt{2} }\sup_{ v_h\in \mathcal{X}_h}\frac{\mathcal{A}(w_h,v_h)}{\|v_h\|_{ \mathcal{Y}}} \quad \forall w_h\in  \mathcal{X}_h.
\end{equation*}
\end{prop} 

\begin{thm}\label{thm:quasi-optimality}
There exists a unique solution $u_h\in \mathcal{X}_h$ to the discrete
problem \eqref{eq:discrete-system}. Moreover, it holds 
\begin{equation*} 
\|u-u_h\|_{\mathcal{X}_h}\leq 5 \inf_{w_h\in \mathcal{X}_h}\|u-w_h \|_\mathcal{X},
\end{equation*}
where $u\in\mathcal{X}$ is the solution of \eqref{eq:var_for}. 
\end{thm}
We have then the following a-priori estimate for  $h$-refinement. 

{
\begin{thm}
  Let $q$ be an integer such that   $1 < q \leq \min\{p_s,p_t
 \}+1$. If    $u\in \mathcal{X}\cap  H^q\left (
   \Omega\times (0,T) \right )$
  is the solution of \eqref{eq:var_for} and $u_h\in\mathcal{X}_h$ is the solution of \eqref{eq:discrete-system},
  then it holds
\begin{equation} 
  \label{eq:a-riori-error-bound}
 \|u-u_h\|_{\mathcal{X}_h}\leq C\sqrt{  \frac{\gamma^2}{\nu} +\nu}\left (h_t^{q-1}
   + h_s^{q-1}\right)   \| u \| _{ H^q\left ( 
   \Omega\times (0,T)\right ) } 
 \end{equation}
  where $C$ is independent of $h_s, h_t, \gamma $, $\nu$ and $u$.
\end{thm}
\begin{proof}
  We  use the
  approximation  estimates of the isogeometric spaces from \cite{Da2012}. We
  report here only the main steps, since the proof is similar to
  the one of \cite[Proposition 4]{Montardini2018space}.

   Let  $ \Pi_h :  L^2\left (  \Omega\times (0,T) \right )
\rightarrow \mathcal{X}_h$ be  a suitable projection, based on a tensor-product construction as in
 \cite{Da2012}, and $L^2\left ( \Omega\times (0,T)\right )  \equiv L^2(0,T;L^2(\Omega))  \equiv L^2(\Omega)  \otimes   L^2(0,T )
$. Then $ \Pi_h   =    \Pi_{s,h_s} \otimes \Pi_{t,h_t}$,
 where $\Pi_{s,h_s} :  L^2(\Omega)  \rightarrow \mathcal{X}_{s,h_s}$  and
 $\Pi_{t,h_t} :  L^2(0,T)  \rightarrow \mathcal{X}_{t,h_t}$ are
 projections on the space and time  components, respectively, of the isogeometric
 space  $\mathcal{X}_h$. The following bounds are
 straightforward  generalizations of \cite[Proposition 4.1, Theorem 5.1]{Da2012}
 \begin{equation}\label{eq:thm2-1}
   \begin{aligned}
       \|\partial_t (u-\Pi_h u)  \|_{ L^2(\Omega)
 \otimes L^2(0,T )    } =  &  |  (u-\Pi_h u)  |_{ L^2(\Omega)
 \otimes  H^1(0,T )  } \\   \leq & C_1\left (h_t^{q-1}
     \| u \| _{   L^2(\Omega) \otimes H^q(0,T) } + h_s^{q-1}
     \| u \| _{   H^{q-1}(\Omega)\otimes H^1(0,T)} \right)
   \\   \leq & C_1\left (h_t^{q-1}
   + h_s^{q-1}\right)   \| u \| _{ H^q\left( \Omega \times (0,T) 
   \right) } 
   \end{aligned}
 \end{equation}
 and 
  \begin{equation}\label{eq:thm2-2}
    \begin{aligned}
        \|u-\Pi_h u \|_{ L^2(0,T;H_0^1(\Omega))} = &     |u-\Pi_h u
        |_{H^{1}(\Omega)  \otimes L^2(0,T) }\\ \leq &  C_2 \left (h_t^{q-1}
     \| u \| _{  H^1(\Omega) \otimes H^{q-1}(0,T) } + h_s^{q-1}
     \| u \| _{ H^{q}(\Omega) \otimes  L^2(0,T) } \right)
   \\   \leq & C_2\left (h_t^{q-1}
   + h_s^{q-1}\right)   \| u \| _{ H^q\left (\Omega \times (0,T) 
   \right ) } 
    \end{aligned}
   \end{equation}
 Therefore, using \eqref{eq:thm2-1} with the obvious bound 
 $ \|\partial_t (u-\Pi_h u)  \|_{L^2(0,T;H^{-1}(\Omega))}  \leq
     \|\partial_t (u-\Pi_h u)  \|_{ L^2(\Omega)
 \otimes L^2(0,T )}$, and
 \eqref{eq:thm2-2},  we get   
 \begin{equation}\label{eq:thm2-3}
  \|u-\Pi_h u \|^2_\mathcal{X}\leq C_3\left( \frac{\gamma^2}{\nu} +\nu\right)\left (h_t^{q-1}
   + h_s^{q-1}\right)^2   \| u \| ^2_{ H^q\left (\Omega\times (0,T) 
   \right ) } 
\end{equation}
and then \eqref{eq:a-riori-error-bound},  thanks to Theorem
\ref{thm:quasi-optimality}.
The constants   $C_1, C_2, C_3$  above are  independent of $h_s, h_t, \gamma, \nu$ and $u$. 
\end{proof}
 
\begin{rmk}
\label{rem:on-the-error-bound}
   In Theorem \ref{thm:quasi-optimality}, the degrees $p_t$,  $p_s$
and the mesh-sizes $h_t$,  $h_s$ play a similar role. This motivates
our   choice $p_t=p_s=:p$  and $h_t= h_s=:h$ for the numerical
tests in Section \ref{sec:numerical-tests}. In this case, and if the
solution $u$ is smooth, \eqref{eq:a-riori-error-bound} yields
$h$-convergence of order $p$.  A sharper error analysis is possible
taking into account a different regularity of the solution $u$ in
space and time, in the line of the anisotropic estimates of
\cite{Da2012}.
 \end{rmk}
  }
\subsection{Discrete system}
The linear system associated to \eqref{eq:discrete-system} is
\begin{equation}
\mathbf{A}\mathbf{u} = \mathbf{f},
\label{eq:sys_solve}
\end{equation}
where  
$[\mathbf{A}]_{i,j}=\mathcal{A}( B_{j, \vect{p}}, B_{i,
  \vect{p}})$ and $[\mathbf{f}]_{i}=\mathcal{F}( B_{i,
  \vect{p}})$. The tensor-product structure of the isogeometric
space \eqref{eq:disc_space} allows to write  the system  matrix $\mathbf{A}$ as sum of Kronecker products of matrices as 
\begin{equation} 
\mathbf{A}  \   =   \gamma   \mathbf{W}_t \otimes  \mathbf{M}_s + \nu  \mathbf{M}_t\otimes  \mathbf{K}_s, \label{eq:syst_mat} 
\end{equation}
where  for $i,j=1,\dots,n_t$
\begin{subequations}
\begin{equation}
\label{eq:time_mat}
 [ \mathbf{W}_t]_{i,j}  = \int_{0}^T   b'_{j,  {p}_t}(t)\,  b_{i,
   {p}_t}(t) \, \dt \quad \text{and} \quad   [ \mathbf{M}_t]_{i,j} = \int_{0}^T\,  b_{i, p_t}(t)\,  b_{j, p_t}(t)  \, \dt , 
  \end{equation}
  while for $  i,j=1,\dots,N_s $
  \begin{equation}
  \label{eq:space_mat}
 [ \mathbf{K}_s]_{i,j}  =  \int_{\Omega} \nabla  B_{i, p_s}(\vect{x})\cdot \nabla  B_{j, p_s}(\vect{x}) \ \d\Omega  \quad \text{and} \quad  [ \mathbf{M}_s]_{i,j}  =  \int_{\Omega}  B_{i, p_s}(\vect{x}) \  B_{j, p_s}(\vect{x}) \ \d\Omega .  
 \end{equation}
 \label{eq:pencils}
\end{subequations}

\section{Preconditioner definition and application}
 
We introduce, for the system \eqref{eq:sys_solve},   the  preconditioner  
\begin{equation*} 
 [\widehat{\mathbf{A}}]_{i,j}:=\widehat{\mathcal{A}}(\widehat{B}_{j, \vect{p}},\widehat{B}_{i, \vect{p}}),
 \end{equation*}
where 
\[ \widehat{\mathcal{A}}(\widehat{v},\widehat{w}) :=
 \int_{0}^1 \int_{\widehat{\Omega}}  \left(    \gamma \partial_t \widehat{v} \,  \widehat{w} + T \nu  \nabla \widehat{v}\cdot\, \nabla \widehat{w}   \right)\,\d\widehat{\Omega}\, \d \tau \quad   \forall\widehat{v},\widehat{w}\in \widehat{\mathcal{X}}_h.\]
 We have  
\begin{equation}
\label{eq:prec_definition}
\widehat{\mathbf{A}}=\gamma {\mathbf{W}}_t\otimes\widehat{\mathbf{M}}_s+\nu{\mathbf{M}}_t\otimes\widehat{\mathbf{K}}_s,
\end{equation}
where $\widehat{\mathbf{K}}_s$ and $\widehat{\mathbf{M}}_s$ are the
equivalent of \eqref{eq:space_mat} in the parametric domain, i.e. we define for $i,j=1,\dots,N_s$ 
  \begin{equation}
  \label{eq:space_mat_par}
 [ \widehat{\mathbf{K}}_s]_{i,j}  =  \int_{\widehat{\Omega}}\nabla  {\widehat{B}}_{i, p_s}(\vect{\eta})\cdot \nabla  {\widehat{B}}_{j, p_s}(\vect{\eta}) \ \d\widehat{\Omega} \quad \text{and} \quad [ \widehat{\mathbf{M}}_s]_{i,j}  =  \int_{\widehat{\Omega}}  \widehat{B}_{i, p_s}(\vect{\eta}) \  \widehat{B}_{j, p_s}(\vect{\eta}) \ \d\widehat{\Omega}.  
 \end{equation}
We emphasize that the time matrices appearing in \eqref{eq:prec_definition} are the same ones appearing in the system matrix \eqref{eq:syst_mat}. This is because for $i,j=1,\dots,n_t$ we have
\[ [{\mathbf{W}}_t]_{i,j} = \int_{0}^T b'_{j,  {p}_t}(t) \ b_{i,  {p}_t}(t) \
\dt = \int_{0}^1 \widehat b'_{j,  {p}_t}(\tau)\  \widehat b_{i,  {p}_t}(\tau) \
\d\tau  \]
and
\[ [{\mathbf{M}}_t]_{i,j} = \int_{0}^T b_{j,  {p}_t}(t) \ b_{i,  {p}_t}(t) \
\dt = T \int_{0}^1 \widehat b_{j,  {p}_t}(\tau)\  \widehat b_{i,  {p}_t}(\tau) \
\d\tau. \]

Thanks to \eqref{eq:tens-prod-space-parametric}, the spatial  matrices \eqref{eq:space_mat_par}  have  the following
structure
\begin{equation}
\label{eq:space_univ_par_mat}
\widehat{\mathbf{K}}_s=\sum_{l=1}^d\widehat{\mathbf{M}}_d\otimes\dots\otimes\widehat{\mathbf{M}}_{l+1}\otimes\widehat{\mathbf{K}}_l\otimes\widehat{\mathbf{M}}_{l-1}\otimes\dots\otimes\widehat{\mathbf{M}}_1 \quad \text{and} \quad \widehat{\mathbf{M}}_s=\widehat{\mathbf{M}}_d\otimes\dots\otimes\widehat{\mathbf{M}}_1,
\end{equation}
where for $l=1,\dots,d$ and for $ i,j=1,\dots,n_{s,l}$ we define
\[
[\widehat{\mathbf{K}}_l]_{i,j}:=\int_0^1\widehat{b}'_{i,p_s}(\eta_k)\widehat{b}'_{j,p_s}(\eta_k) \d\eta_k \quad \text{and} \quad [\widehat{\mathbf{M}}_l]_{i,j}:=\int_0^1\widehat{b}_{i,p_s}(\eta_k)\widehat{b}_{j,p_s}(\eta_k) \d\eta_k.
\]

The efficient application of the proposed preconditioner, that is, the
solution of a linear system with matrix $ \widehat{\mathbf{A}}$, should
exploit  the structure highlighted above.   When the
pencils $ ( \mathbf{W}_t,\mathbf{M}_t )$,  $ ( \widehat{\mathbf{K}}_1,\widehat{\mathbf{M}}_1 ),
\ldots,  ( \widehat{\mathbf{K}}_d,\widehat{\mathbf{M}}_d  )$ 
admit a stable generalized eigendecomposition,   a possible approach is
the Fast Diagonalization (FD) method, see \cite{Deville2002} and
\cite{Lynch1964}  for details.
We will see in Section \ref{sec:stable_space} that the  spatial pencils $ ( \widehat{\mathbf{K}}_1,\widehat{\mathbf{M}}_1  ),
\ldots,  ( \widehat{\mathbf{K}}_d,\widehat{\mathbf{M}}_d  )$ admit a stable diagonalization, but this is not the case of $ ( \mathbf{W}_t,\mathbf{M}_t )$, that needs a
  special treatment  as explained in  Section \ref{sec:stable_time}.

\subsection{Stable factorization of the pencils $ (\widehat{\mathbf{K}}_l, \widehat{\mathbf{M}}_l)$ for $  l=1,\dots,d$}
\label{sec:stable_space}

The spatial stiffness and mass matrices $\widehat{\mathbf{K}}_l$ and $\widehat{\mathbf{M}}_l$ are symmetric and positive definite for  $l=1,\dots,d$. Thus, the pencils 
  $ (\widehat{\mathbf{K}}_l, \widehat{\mathbf{M}}_l )$ for  $l=1,\dots,d$ admit 
    the  generalized eigendecomposition  \[
    \widehat{\mathbf{K}}_l\mathbf{U}_l =
\widehat{\mathbf{M}}_l\mathbf{U}_l\mathbf{\Lambda}_l,
\] where the matrices $
\mathbf{U}_l$ contain  in each column the
$\widehat{\mathbf{M}}_l$-orthonormal  generalized eigenvectors and 
  $\mathbf{\Lambda}_l$ are diagonal matrices whose entries contain the generalized eigenvalues.
Therefore we have for $l=1,\dots ,d$ the  factorizations 
\begin{equation}
\label{eq:space_eig}
\mathbf{U}^T_l \widehat{\mathbf{K}}_l \mathbf{U}_l=
\mathbf{\Lambda}_l  \quad\text{ and  }\quad \mathbf{U}^T_l \widehat{\mathbf{M}}_l \mathbf{U}_l=
\mathbb{I}_{  n_{s,l}},
\end{equation}
 where $\mathbb{I}_{n_{s,l}}$ denotes the identity matrix of dimension $n_{s,l} \times
n_{s,l}$. 
The stability of the decomposition \eqref{eq:space_eig} is
expressed by the condition number of the eigenvector matrix.  In
particular  $\mathbf{U}^T_l \widehat{\mathbf{M}}_l \mathbf{U}_l=
\mathbb{I}_{  n_{s,l}}$ implies  that
 \[ \kappa_2
(\mathbf{U}_l) := \|\mathbf{U}_l\|_{2 }
  \|\mathbf{U}_l^{-1}\|_{2 }  =
  \sqrt{\kappa_2(\widehat{\mathbf{M}}_l)},\]  where  $   \|\cdot\|_{ 2
  }$ is the norm induced by the Euclidean vector norm. 
  The condition number $\kappa_2(\widehat{\mathbf{M}}_l)$ has been
  studied in \cite{gahalaut2014condition} and it does not depend on
    the mesh-size,  but it depends on the polynomial degree.
 Indeed,  we report  in Table
  \ref{tab:cond_number_space} the   behavior of $\kappa_2
(\mathbf{U}_l)$   for different values of  spline degree   $p_s$ and for different uniform discretizations  with  number of elements denoted by $n_{el}$.      We observe  that $\kappa_2
(\mathbf{U}_l)$  exhibits a   dependence only on  
$p_s$, but stays moderately low for all low polynomial degrees that
are in the range of interest. 
{\renewcommand\arraystretch{1.2} 
\begin{table}[H]                                                        
	\centering                                                              
	\begin{tabular}{|c|c|c|c|c|c|c|c|}                                      
		\hline                                                                  
		 $n_{el}$ & $p_s=2$ & $p_s=3$ & $p_s=4$ & $p_s=5$ & $p_s=6$ & $p_s=7$ & $p_s=8$ \\                                              
		\hline                                                                  
		\z\z32 & $2.7 \cdot 10^0 $ & $ 4.5 \cdot 10^0$ & $ 7.6 \cdot 10^0$ & $ 1.3 \cdot 10^1$ & $ 2.1  \cdot 10^1$ & $ 3.5  \cdot 10^1$ & $ 5.7  \cdot 10^1$ \\   
		\hline                                                          
		\z\z64  & $ 2.7 \cdot 10^0 $ & $ 4.5 \cdot 10^0$ & $ 7.6 \cdot 10^0$ & $ 1.3  \cdot 10^1$ & $ 2.1  \cdot 10^1$ & $ 3.5 \cdot 10^1$ & $ 5.7  \cdot 10^1$ \\   
		\hline                                                          
		\z128 & $ 2.7 \cdot 10^0$ & $ 4.5 \cdot 10^0$ & $ 7.6 \cdot 10^0$ & $ 1.3 \cdot 10^1 $ & $ 2.1  \cdot 10^1$ & $ 3.5  \cdot 10^1$ & $ 5.7 \cdot 10^1$ \\   
		\hline                                                          
		\z256  & $ 2.7 \cdot 10^0$ & $ 4.5 \cdot 10^0$ & $ 7.6 \cdot 10^0$ & $ 1.3  \cdot 10^1$ & $ 2.1  \cdot 10^1$ & $ 3.5  \cdot 10^1$ & $ 5.7  \cdot 10^1$ \\  
		\hline                                                          
		\z512  & $ 2.7 \cdot 10^0$ & $ 4.5 \cdot 10^0$ & $ 7.6 \cdot 10^0$ & $ 1.3  \cdot 10^1$ & $ 2.1  \cdot 10^1$ & $ 3.5  \cdot 10^1$ & $ 5.7 \cdot 10^1$ \\   
		\hline                                                          
		1024  & $ 2.7 \cdot 10^0$ & $ 4.5 \cdot 10^0$ & $ 7.6 \cdot 10^0$ & $ 1.3  \cdot 10^1$ & $ 2.1  \cdot 10^1$ & $ 3.5  \cdot 10^1$ & $ 5.7 \cdot 10^1 $ \\ 
		\hline                                                                  
	\end{tabular}                                                           
\caption{$\kappa_2 (\mathbf{U}_l)$ for different polynomial degrees $p_s$ and number of elements $n_{el}$.}                                                                               
\label{tab:cond_number_space}               
\end{table}}

\subsection{Stable factorization of the pencil $ (\mathbf{W}_t, \mathbf{M}_t )$}
\label{sec:stable_time}
 \subsubsection{Numerical instability of the eigendecomposition} 

While $\mathbf{M}_t$ is symmetric,  $\mathbf{W}_t$ is neither symmetric nor skew-symmetric. 
Indeed
\begin{equation}
\label{eq:integral_W}
\begin{aligned}[t]
  [\mathbf{W}_t]_{i,j}  +[\mathbf{W}_t]_{j,i}   =   \!
 \! \int_{0}^T   b'_{j,  {p}_t}(t)\     b_{i,  {p}_t}(t) \
\dt + \! \!  \int_{0}^T  b'_{i,  {p}_t}(t)\    b_{j,
  {p}_t}(t) \ \dt    =        b_{i,  {p}_t}(T)   \   b_{j,  {p}_t}(T)
\end{aligned}
\end{equation} 
where $   b_{i,  {p}_t}(T) \  b_{j,
  {p}_t}(T)$ vanishes for all  $ i=1,\dots, n_t-1$ or  $ j=1,\dots, n_t-1$. 
A numerical computation of  the   generalized eigendecomposition of the pencil $ (
  \mathbf{W}_t,\mathbf{M}_t )$, that is 
\begin{equation}\label{eq:diag}
\mathbf{W}_t \mathbf{U} = \mathbf{M}_t \mathbf{U} \mathbf{\Lambda}_t,
\end{equation}
where $\mathbf{\Lambda}_t$  is  the diagonal matrix of the generalized complex
eigenvalues and ${\mathbf{U}} $ is the  complex matrix whose columns are
the generalized eigenvectors normalized w.r.t. the norm induced by 
 ${\mathbf{M}_t}$,   reveals that the
eigenvectors are far from $\mathbf{M}_t$-orthogonality, i.e. the matrix    ${\mathbf{U}}^*\mathbf{M}_t 
{\mathbf{U}}$ is not diagonal.  
  We set $T=1$ and we report in Table \ref{tab:cond_number} 
the  condition number 
  $ \kappa_2 (\mathbf{U})$ 
   for different values of spline degree $p_t$ and for different  uniform discretizations with   $n_{el}$ number of elements.   In contrast to the spatial case (see  Section \ref{sec:stable_space}),   $ \kappa_2 (\mathbf{U})$
is large and grows
exponentially with respect to the spline degree $p_t$ and the level of mesh
refinement.  
 This test  clearly indicates a
numerical instability when computing the generalized eigendecomposition of $(\mathbf{W}_t, \mathbf{M}_t)$.
A similar behavior has  also been highlighted in \cite{hofer2019parallel}.  
{\renewcommand\arraystretch{1.2}  
\begin{table}[H]                        \centering                             
{ \begin{tabular}{|c|c|c|c|c|c|c|c|}      \hline                                  $n_{el}$ & $p_t=2$ & $p_t=3$ & $p_t=4$ & $p_t=5$ & $p_t=6$ & $p_t=7$ & $p_t=8$ \\                                                           
\hline                                  \z\z32 & $8.9 \cdot 10^2 $ & $ 3.0 \cdot10^4 $ & $ 5.0 \cdot10^{4\z} $ & $ 3.4 \cdot10^{5\z} $ & $ 3.1 \cdot10^{6\z} $ & $ 4.2 \cdot10^{7\z} $ & $ 7.0 \cdot10^{8\z}$ \\     
\hline                                                                                    
\z\z64  & $ 4.4 \cdot 10^3$ & $ 2.6 \cdot 10^5 $ & $ 5.0 \cdot 10^{5\z} $ & $ 5.4 \cdot10^{6\z} $ & $ 8.9 \cdot10^{7\z} $ & $ 3.1 \cdot10^{9\z} $ & $ 2.0 \cdot10^{10}$ \\      
\hline                                                                                    
\z128 & $ 2.3 \cdot 10^4 $ & $ 1.2 \cdot 10^6 $ & $ 5.8 \cdot 10^{6\z} $ & $ 1.0 \cdot 10^{8\z} $ & $ 3.0 \cdot 10^{9\z} $ & $ 6.4 \cdot 10^{11} $ & $ 1.3 \cdot 10^{12}$ \\ 
\hline                                                                                    
\z256  & $ 1.2 \cdot 10^5 $ & $ 9.4 \cdot10^6 $ & $ 7.6 \cdot 10^{7\z} $ & $ 2.1 \cdot 10^{9\z} $ & $ 1.2 \cdot 10^{11} $ & $ 1.2 \cdot 10^{13} $ & $ 2.1 \cdot 10^{13}$ \\
\hline                                                                                    
\z512  & $ 7.0 \cdot 10^5 $ & $ 8.3 \cdot10^7 $ & $ 1.1 \cdot 10^{9\z} $ & $ 4.9 \cdot 10^{10} $ & $ 4.5 \cdot 10^{12} $ & $ 3.6 \cdot 10^{13} $ & $ 4.9 \cdot 10^{12}$ \\
\hline                                                                                    
1024  & $ 4.1 \cdot 10^6 $ & $ 8.0 \cdot 10^8 $ & $ 1.9 \cdot 10^{10} $ & $ 1.3 \cdot 10^{12} $ & $ 9.6 \cdot 10^{12} $ & $ 1.4 \cdot 10^{12 }$ & $ 5.6 \cdot 10^{12}$ \\
\hline                                  \end{tabular}                          } \caption{
$\kappa_2 (\mathbf{U})$ for different degree $p_t$ and number of elements $n_{el}$.}                                                                               
\label{tab:cond_number}            \end{table}}

 \subsubsection{Construction of the stable factorization} 
The analysis above  motivates the search of a different but  stable factorization of the pencil
  $(\mathbf{W}_t,  \mathbf{M}_t)$.  We  look now for a factorization of the form 
\begin{equation}\label{eq:arrow}
\mathbf{W}_t \mathbf{U}_t = \mathbf{M}_t \mathbf{U}_t
\mathbf{\Delta}_t,
\end{equation}
where $ \mathbf{\Delta}_t$ is a  complex arrowhead matrix, i.e. with   non-zero entries 
allowed on  the diagonal, on the last row and on the last column only.
We also require that  $\mathbf{U}_t $ 
 fulfils the orthogonality condition
\begin{equation}
  \label{eq:Mt-orthogonality}
\mathbf{U}_t^*\mathbf{M}_t {\mathbf{U}_t} = \mathbb{I}_{n_t}.
\end{equation} 
From \eqref{eq:arrow}--\eqref{eq:Mt-orthogonality} we then  obtain the factorizations
\begin{equation}
\label{eq:time_eig}
  \mathbf{U}_t^*  \mathbf{W}_t \mathbf{U}_t =
\mathbf{\Delta}_t  \quad \text{ and  }\quad \mathbf{U}_t^*  \mathbf{M}_t \mathbf{U}_t =
\mathbb{I}_{  n_t}. \end{equation}
With this aim, we look for $\mathbf{U}_t$ as follows:
\begin{equation}
 \label{eq:U_t}
 \mathbf{U}_t :=\begin{bmatrix}
\overset{\circ}{\mathbf{U}}_t & \mathbf{r}\\[2pt]
 \boldsymbol{0}^T & \rho
\end{bmatrix}
 \end{equation}
where $\overset{\circ}{\mathbf{U}}_t  \in\mathbb{C}^{(n_t-1)\times
  (n_t-1)}$, $\mathbf{r}\in\mathbb{C}^{n_t-1}$, $\rho\in \mathbb{C} $
and where $\vect{0}\in\mathbb{R}^{n_t-1}$ denotes the null vector. In
order to guarantee the non-singularity of $\mathbf{U}_t$, we further
impose $\rho\neq 0$. Accordingly,  we split the time matrices $\mathbf{W}_t$ and $\mathbf{M}_t$ as
 \begin{equation}\label{eq:split}
\mathbf{W}_t = \begin{bmatrix}
\overset{\circ}{\mathbf{W}}_t & \mathbf{w}\\[2pt]
-\mathbf{w}^T & \omega
\end{bmatrix}\quad\text{and}\quad
\mathbf{M}_t=\begin{bmatrix}
\overset{\circ}{\mathbf{M}}_t & \mathbf{m}\\[2pt]
\mathbf{m}^T & \mu
\end{bmatrix}, 
\end{equation}
where we have defined   
\[\omega:= [\mathbf{W}_t]_{n_t,n_t}, \qquad \mu:= [\mathbf{M}_t]_{n_t,n_t},
\]\[ 
[\mathbf{w}]_i= [\mathbf{W}_t]_{i,n_t}  \quad\text{and}\quad   [\mathbf{m}]_i= [\mathbf{M}_t]_{i,n_t}   \quad\text{for}\quad  i=1,\dots, n_t-1, 
\]
\[ 
\z[\overset{\circ}{\mathbf{W}}_t]_{i,j} = [\mathbf{W}_t]_{i,j}    \quad\text{and}  \quad [\overset{\circ}{\mathbf{M}}_t]_{i,j} = [\mathbf{M}_t]_{i,j}   \quad \text{for}   \quad i,j=1,\dots, n_t-1.
\] 
  Recalling \eqref{eq:integral_W}, we observe that
$\overset{\circ}{ \mathbf{W}}_t$ is skew-symmetric and, since  $\overset{\circ}{\mathbf{M}}_t$ is symmetric, we can  write the  eigendecomposition of the pencils $(\overset{\circ}{\mathbf{W}}_t,  \overset{\circ}{\mathbf{M}}_t)$:
\begin{equation}
\label{eq:factorization_first_block}
\overset{\circ}{\mathbf{W}}_t \overset{\circ}{\mathbf{U}}_t =
\overset{\circ}{\mathbf{M}}_t \overset{\circ}{\mathbf{U}}_t
\overset{\circ}{\mathbf{\Lambda}}_t \quad  \text{ with } \quad
\overset{\circ \z}{\mathbf{U}_t^*}\overset{\circ}{\mathbf{M}}_t \overset{\circ}{\mathbf{U}}_t = \mathbb{I}_{n_t-1}, 
\end{equation}
where  $\overset{\circ}{\mathbf{U}}_t$ contains the complex
generalized eigenvectors and   $\overset{\circ}{\mathbf{\Lambda}}_t$ is
the diagonal matrix of the  generalized eigenvalues, that are pairs of
complex conjugate pure imaginary numbers plus, eventually, the
eigenvalue zero. From \eqref{eq:U_t}--\eqref{eq:split},  it follows 
\[
\mathbf{U}_t^*\mathbf{M}_t\mathbf{U}_t=\begin{bmatrix}
\mathbb{I}_{n_t-1}& \overset{\circ \z}{\mathbf{U}^*_t}\overset{\circ}{\mathbf{M}}_t\mathbf{r}+\overset{\circ \z}{\mathbf{U}^*_t}\mathbf{m}\rho\\
\mathbf{r}^*\overset{\circ}{\mathbf{M}}_t\overset{\circ}{\mathbf{U}}_t+\rho^*\mathbf{m}^T \overset{\circ}{\mathbf{U}}_t& \left[ \mathbf{r}^* \rho^* \right]\mathbf{M}_t \begin{bmatrix}
\mathbf{r}\\ \rho
\end{bmatrix}
\end{bmatrix},
\]
where for the top-left block we have used \eqref{eq:factorization_first_block}.

The orthogonality condition in \eqref{eq:Mt-orthogonality}  holds if and only if $\mathbf{r}$ and $\rho$ fulfil the two conditions:
 
\begin{subnumcases}{}
\overset{\circ \z}{\mathbf{U}^*_t}\overset{\circ}{\mathbf{M}}_t\mathbf{r}+\overset{\circ \z}{\mathbf{U}^*_t}\mathbf{m}\rho = \vect{0}, \label{eq:syst_all-1}\\
 \left[ \mathbf{r}^* \rho^* \right]\mathbf{M}_t \begin{bmatrix}
\mathbf{r}\\ \rho
\end{bmatrix} =1. \label{eq:syst_all-2}
\end{subnumcases}

In order to compute $\mathbf{r}$ and $\rho$, we first  find
$\mathbf{v}\in \mathbb{C}^{n_t-1}$ such that 
\begin{equation}\label{eq:syst_v}
 \overset{\circ}{\mathbf{M}}_t\mathbf{v}=- \mathbf{m};
 \end{equation}
then we normalize the vector  $\begin{bmatrix}
 \mathbf{v} \\ 1
 \end{bmatrix}$ w.r.t. the $\|\cdot\|_{\mathbf{M}_t}$-norm
 to get 
\[\begin{bmatrix}
\mathbf{r}\\ \rho
\end{bmatrix}:=\frac{\begin{bmatrix}
\mathbf{v}\\ 1
\end{bmatrix}}{  \left( [\mathbf{v}^*\ 1]\mathbf{M}_t\begin{bmatrix}
\mathbf{v}\\ 1
\end{bmatrix}\right)^{\tfrac12}  }
\]
that fulfils \eqref{eq:syst_all-1}--\eqref{eq:syst_all-2}.
Finally, we get  \eqref{eq:arrow} by defining
\begin{equation}
\label{eq:time_eig_2}
  \mathbf{\Delta}_t:=\mathbf{U}_t^*\mathbf{W}_t \mathbf{U}_t   =  \begin{bmatrix}
\overset{\circ}{\mathbf{\Lambda}}_t& \mathbf{g}\\[2pt]
  -\mathbf{g}^* 
  & \sigma
\end{bmatrix},
\end{equation}
where $\mathbf{g}:=\overset{\circ}{{\mathbf{U}}_t^*}\left[\overset{\circ}{\mathbf{W}}_t \ \mathbf{w}\right]\begin{bmatrix}
\mathbf{r}\\ \rho
\end{bmatrix}$ and $\sigma:=\left[ {\mathbf{r}}^*   {\rho}^*\right] \mathbf{W}_t \begin{bmatrix}
\mathbf{r}\\ \rho
\end{bmatrix}$. Note that matrix \eqref{eq:time_eig_2} has an arrowhead structure.

  To assess the stability of the new decomposition \eqref{eq:time_eig}, 
we set $T=1$ and we compute the  condition number  $\kappa_2(\mathbf{U}_t)$ for different values of spline degree  $p_t$    and for various  uniform discretizations with number of   elements $n_{el}$.  
 Thanks to \eqref{eq:Mt-orthogonality}, we have
 $\kappa_2(\mathbf{U}_t)=\sqrt{\kappa_2(
   {\mathbf{M}_t)}}$. The results, reported in Table
 \ref{tab:cond_number_space_new}, show that  the
 condition numbers  $\kappa_2(\mathbf{U}_t)$ are uniformly bounded
 w.r.t. the mesh refinement, they grow with respect to the polynomial
 degree but they are moderately small for all the degrees of interest. 
 We conclude that the factorization \eqref{eq:time_eig} for the time
 pencil $(\mathbf{W}_t, \mathbf{M}_t)$ is stable. 

{\renewcommand\arraystretch{1.1} 
  \begin{table}[h]                                                        
 	\centering                                                              
 	\begin{tabular}{|c|c|c|c|c|c|c|c|}                                      
 		\hline                                                                  
 		$n_{el}$ & $p_t=2$ & $p_t=3$ & $p_t=4$ & $p_t=5$ & $p_t=6$ & $p_t=7$ & $p_t=8$ \\                                              
 		\hline                                                                  
 		\z\z32 & $3.2 \cdot 10^0 $ & $ 5.2 \cdot 10^0$ & $ 8.3 \cdot 10^0$ & $ 1.3 \cdot 10^1$ & $ 2.2  \cdot 10^1$ & $ 3.6  \cdot 10^1$ & $ 5.9  \cdot 10^1$ \\   
 		\hline                                                          
 		\z\z64  & $3.3 \cdot 10^0 $ & $ 5.2 \cdot 10^0$ & $ 8.3 \cdot 10^0$ & $ 1.3 \cdot 10^1$ & $ 2.2  \cdot 10^1$ & $ 3.6  \cdot 10^1$ & $ 5.9  \cdot 10^1$ \\   
 		\hline                                                          
 		\z128 & $3.3 \cdot 10^0 $ & $ 5.2 \cdot 10^0$ & $ 8.3 \cdot 10^0$ & $ 1.3 \cdot 10^1$ & $ 2.2  \cdot 10^1$ & $ 3.6  \cdot 10^1$ & $ 5.9  \cdot 10^1$ \\   
 		\hline                                                          
 		\z256  & $3.3 \cdot 10^0 $ & $ 5.2 \cdot 10^0$ & $ 8.3 \cdot 10^0$ & $ 1.3 \cdot 10^1$ & $ 2.2  \cdot 10^1$ & $ 3.6  \cdot 10^1$ & $ 5.9  \cdot 10^1$ \\  
 		\hline                                                          
 		\z512  & $3.3 \cdot 10^0 $ & $ 5.2 \cdot 10^0$ & $ 8.3 \cdot 10^0$ & $ 1.3 \cdot 10^1$ & $ 2.2  \cdot 10^1$ & $ 3.6  \cdot 10^1$ & $ 5.9  \cdot 10^1$ \\   
 		\hline                                                          
 		1024  & $3.3 \cdot 10^0 $ & $ 5.2 \cdot 10^0$ & $ 8.3 \cdot 10^0$ & $ 1.3 \cdot 10^1$ & $ 2.2  \cdot 10^1$ & $ 3.6  \cdot 10^1$ & $ 5.9  \cdot 10^1$ \\ 
 		\hline                                                                  
 	\end{tabular} 
 	\caption{$\kappa_2 (\mathbf{U}_t)$ for different degrees $p_t$ and number of elements $n_{el}$.}                                                                               
 	\label{tab:cond_number_space_new}                                             
 \end{table}}

\subsection{Preconditioner application}
\label{sec:prec_appl}
The application of the preconditioner involves the solution of the linear system 
\begin{equation}
\label{eq:prec_appl}
 \widehat{\mathbf{A}}\mathbf{s}=\mathbf{r},
\end{equation}
where $\widehat{\mathbf{A}}$ has the   structure
\eqref{eq:prec_definition}.
We are able to efficiently solve system \eqref{eq:prec_appl} by extending the FD method. 
The starting points, that are involved in the setup of the preconditioner,  are the following ones:
\begin{itemize}
\item   for  the  pencils
  $(\widehat{\mathbf{K}}_l, \widehat{\mathbf{M}}_l)$ for  $l=1,\dots,d$
  we have   the factorizations \eqref{eq:space_eig};
\item for the pencil  $(\mathbf{W}_t,
  \mathbf{M}_t)$ 
 we have the factorization \eqref{eq:time_eig}.
\end{itemize}
Then, by defining $
\mathbf{U}_s:=\mathbf{U}_d\otimes\dots\otimes\mathbf{U}_1$ and $\mathbf{\Lambda}_s:=\sum_{l=1}^d \mathbb{I}_{n_{s,d}}\otimes\dots\otimes\mathbb{I}_{n_{s,l+1}} \otimes
  \mathbf{\Lambda}_l \otimes \mathbb{I}_{n_{s,l-1}}\otimes\dots\otimes\mathbb{I}_{n_{s,1}}$,    we have for the matrix $\widehat{\mathbf{A}}$ the factorization 
\begin{equation}\label{eq:preconditioner_str1}
\widehat{\mathbf{A}} =\left(\mathbf{U}_t^*\otimes \mathbf{U}_{s}^T\right)^{-1}\left( { \gamma}\mathbf{\Delta}_t \otimes \mathbb{I}_{N_s}+ { \nu}{\mathbb{I}}_{n_t} \otimes  {\mathbf{\Lambda}}_s\right)\left(\mathbf{U}_t\otimes \mathbf{U}_{s}\right)^{-1} .
\end{equation}
Note that the second factor in \eqref{eq:preconditioner_str1} has the
block-arrowhead structure 
 \begin{equation}
  \begin{aligned}
    {  \gamma}\mathbf{\Delta}_t \otimes \mathbb{I}_{N_s}+  {  \nu}{\mathbb{I}}_{n_t}
    \otimes  {\mathbf{\Lambda}}_s & =
\begin{bmatrix}
\mathbf{H}_1 &  &  & \mathbf{B}_1 \\[2pt]
   &\ddots & &\vdots \\[2pt]
  \quad &  & \mathbf{H}_{n_t-1} &   \mathbf{B}_{n_t-1} \\[2pt]
-\mathbf{B}^*_1 &  \ldots & -\mathbf{B}^*_{n_t-1} & \mathbf{H}_{n_t}
\end{bmatrix}
  \end{aligned}
\label{eq:syst-arr}
\end{equation}
where  $\mathbf{H}_i$ and $\mathbf{B}_i$
 are diagonal matrices defined as 
 \[\mathbf{H}_i:={  \gamma}[\mathbf{\Lambda}_t]_{i,i}\mathbb{I}_{N_s}+{  \nu}\mathbf{\Lambda}_s\quad \text{ and }
 \quad \mathbf{B}_i:={  \gamma}[\mathbf{g}]_i\mathbb{I}_{N_s}  \quad  
  \text{ for }\quad i=1,\dots,n_t-1, \]  \[ \mathbf{H}_{n_t}:={  \gamma}\sigma{\mathbb{I}}_{N_s}+ { \nu}\mathbf{\Lambda}_s. \] 
The matrix \eqref{eq:syst-arr} has   the following easy-to-invert block LU decomposition 
\begin{align}
\label{eq:block_LU} 
 \gamma\mathbf{\Delta}_t \otimes & \mathbb{I}_{N_s}  +  {  \nu}{\mathbb{I}}_{n_t}
    \otimes  {\mathbf{\Lambda}}_s \\&  = \begin{bmatrix}
\mathbb{I}_{N_s} & & &\\
 & \ddots& & \\
 & & \mathbb{I}_{N_s} & \\
 -\mathbf{B}^*_1 \mathbf{H}_1^{-1} & \dots & -\mathbf{B}^*_{n_t-1}\mathbf{H}_{n_t-1}^{-1}  & \mathbb{I}_{N_s}
\end{bmatrix}
\begin{bmatrix}
\mathbf{H}_1 & & & \mathbf{B}_1\\
& \ddots & & \vdots \\
&  &\mathbf{H}_{n_t-1}  & \mathbf{B}_{n_t-1}\\
& & & \mathbf{S}
\end{bmatrix}\nonumber
\end{align}
  where
 $\mathbf{S}:=\mathbf{H}_{n_t}+\sum_{i=1}^{n_t-1}
\mathbf{B}^*_i\mathbf{H}_i^{-1}\mathbf{B}_i$ is a diagonal matrix.

Summarising, the solution of \eqref{eq:prec_appl} can be computed  by the following algorithm.
 
\begin{algorithm}[H]
\caption{Extended FD}\label{al:direct_P}
\begin{algorithmic}[1]
\State Compute the factorizations \eqref{eq:space_eig} and \eqref{eq:time_eig}. 
\State Compute $\widetilde{\mathbf{s}} = (\mathbf{U}_t^*\otimes \mathbf{U}_s^T )\mathbf{s} $.
\State Compute $\widetilde{\mathbf{q}} = \left( {  \gamma}\mathbf{\Delta}_t \otimes \mathbb{I}_{N_s}+ {  \nu} {\mathbb{I}}_{n_t} \otimes  {\mathbf{\Lambda}}_s\right)^{-1} \widetilde{\mathbf{s}}. $
\State Compute $\mathbf{r} = (\mathbf{U}_t\otimes \mathbf{U}_s)\ \widetilde{\mathbf{q}}. $
\end{algorithmic}
\end{algorithm}

\subsection{Preconditioner robustness: partial inclusion of the geometry}
\label{sec:prec_geo}
  The  preconditioner \eqref{eq:prec_definition} does not
  incorporate any information on the spatial parametrization
  $\vect{F}$. 
Thus, the quality of the preconditioning strategy may depend on the geometry map: we see this trend in the  numerical tests presented in  the upper tables of
  Table \ref{tab:rev_quarter} and Table \ref{tab:hollow_torus} of Section \ref{sec:numerical-tests} .
 However, we can generalize \eqref{eq:prec_definition}  by  including
  in the univariate spatial matrices $ \widehat{\mathbf{K}}_l, \widehat{\mathbf{M}}_l $ for $l=1,\dots,d$   a suitable approximation of $\vect{F}$,  
without increasing the asymptotic computational cost. A similar approach has been used also in
\cite{Montardini2018} for the Stokes problem and in  \cite{Montardini2018space} for a least-squares formulation of the heat equation.   We briefly give an overview
of this strategy. 

Referring to Section \ref{sec:iso_space} for the notation of the basis functions, we rewrite the entries of the system matrix \eqref{eq:sys_solve}    in the parametric domain as
\begin{align*} 
 [\mathbf{A}]_{i,j}& =\mathcal{A}({B}_{j,\vect{p}} , {B}_{i,\vect{p}} )\nonumber\\ &=  \gamma \int_0^1\! \int_{\widehat{\Omega}} \tfrac1T    \partial_{\tau}{\widehat{B}}_{ {j}, \vect{p}} {\widehat{B}}_{ {i}, \vect{p}}   | \text{det}(J_{\vect{G}})|\,   \d\widehat{\Omega}\,\d\tau + \int_0^1\! \int_{\widehat{\Omega}} \nu (\nabla {\widehat{B}}_{ {j}, \vect{p}})^T J_{\vect{G}}^{-1}J_{\vect{G}}^{-T} \nabla  {\widehat{B}}_{ {i}, \vect{p}}  | \text{det}(J_{\vect{G}})|\,\d\widehat{\Omega}\,\d\tau \nonumber \\
& ={ \int_0^1\! \int_{\widehat{\Omega}}     \left[  (\nabla{\widehat{B}}_{ {j}, \vect{p}} )^T\ \ {   \partial_{\tau} }{\widehat{B}}_{ {j}, \vect{p}}  \right] \begin{bmatrix} \nu T\ \mathbb{I}_{d} & \\ & \gamma \end{bmatrix} \mathfrak{C}   \left[(\nabla{\widehat{B}}_{ {i}, \vect{p}})^T   \ \ {\widehat{B}}_{ {i}, \vect{p}}   \right]^T   \d\widehat{\Omega}\,\d\tau,}
\end{align*} 
where
\[ \mathfrak{C} :=  \begin{bmatrix}
 J_{\vect{F}}^{-1}J_{\vect{F}}^{-T}|\det(J_{\vect{F}})| & \\
 & | \text{det}(J_{\vect{F}})|   \end{bmatrix} \]
 and where we used that    ${B}_{i,\vect{p}} = \widehat{B}_{i,\vect{p}}\circ\vect{G}^{-1}$, ${B}_{j,\vect{p}} = \widehat{B}_{j,\vect{p}}\circ\vect{G}^{-1} $ and $| \text{det}(J_{\vect{G}})|=T| \text{det}(J_{\vect{F}})|$.   
The construction of the preconditioner is based 
on the following approximation of the diagonal entries only of $\mathfrak{C}$:

 \begin{subequations}
 \label{eq:separation}
 \begin{align}
 & [ \mathfrak{C}(\vect{\eta})]_{l,l}  \approx  [\widetilde{\mathfrak{C}}(\vect{\eta})]_{l,l} :={ \varphi_1(\eta_1)\dots \varphi_{l-1}(\eta_{l-1})\Phi_{l}(\eta_{l})\varphi_{l+1}(\eta_{l+1})\dots\varphi_d(\eta_d)} \quad \text{for }   l=1,\dots,d, \\  
& [ \mathfrak{C}(\vect{\eta})]_{d+1,d+1} \approx  [\widetilde{\mathfrak{C}}(\vect{\eta})]_{d+1,d+1}:={ \varphi_1(\eta_1) \dots\varphi_d(\eta_d). }\z
 \end{align}
\end{subequations}

In order to compute such  an   approximation, we interpolate the functions  $[\widetilde{\mathfrak{C}}(\vect{\eta})]_{l,l}$ in \eqref{eq:separation}   by piecewise constants in each element and we build the univariate factors $\varphi_l$ and $\Phi_l$ by using the separation of variables algorithm detailed in \cite[Appendix C]{Montardini2018space}.
The computational cost of the approximation above is proportional to the number of elements   in $\Omega$, that, when using smooth B-splines, is almost equal to $N_{s}$, independent of $p_s$ and $p_t$ and thus negligible in the whole iterative strategy.

 Then we define 
\[
 [\widetilde{ \mathbf{A}}]_{i,j}:= \int_0^1\! \int_{\widehat{\Omega}}     \left [ (\nabla{\widehat{B}}_{ {j}, \vect{p}}) ^T   \ \      \partial_{\tau}{\widehat{B}}_{ {j}, \vect{p}}  \right] \begin{bmatrix} \nu T \ \mathbb{I}_d & \\ & \gamma \end{bmatrix} \widetilde{\mathfrak{C}}   \left[   (\nabla{\widehat{B}}_{ {i}, \vect{p}})^T  \ \   {\widehat{B}}_{ {i}, \vect{p}} \right]^T\,   \d\widehat{\Omega}\,\d\tau.
 \]
 
 The previous matrix maintains the same  Kronecker structure as \eqref{eq:prec_definition}. Indeed we have that 
\begin{equation}
\label{eq:over_A}
\widetilde{ \mathbf{A} } = \gamma \mathbf{W}_t\otimes\widetilde{\mathbf{M}}_s+ \nu \mathbf{M}_t\otimes\widetilde{\mathbf{K}}_s,
\end{equation}
where  
\begin{equation*} 
\widetilde{\mathbf{K}}_s:=\sum_{l=1}^d\widetilde{\mathbf{M}}_d\otimes\dots\otimes\widetilde{\mathbf{M}}_{l+1}\otimes\widetilde{\mathbf{K}}_l\otimes\widetilde{\mathbf{M}}_{l-1}\otimes\dots\otimes\widetilde{\mathbf{M}}_1, \qquad \widetilde{\mathbf{M}}_s:=\widetilde{\mathbf{M}}_d\otimes\dots\otimes\widetilde{\mathbf{M}}_1,
\end{equation*}
and where  for $l=1,\dots,d$ and for $ i,j=1,\dots,n_{s,l}$  we define
\[
[\widetilde{\mathbf{K}}_l]_{i,j}:=\int_0^1{  \Phi_{l}}(\eta_l)\widehat{b}'_{i,p_s}(\eta_l)\widehat{b}'_{j,p_s}(\eta_l) \d\eta_l  
\ \text{ and } \
  [\widetilde{\mathbf{M}}_l]_{i,j}:=\int_0^1{ \varphi_{l}}(\eta_l)\widehat{b}_{i,p_s}(\eta_l)\widehat{b}_{j,p_s}(\eta_l) \d\eta_l.
\] 
We remark that the application of   \eqref{eq:over_A} can still be performed by Algorithm \ref{al:direct_P}.
Finally, we apply a diagonal scaling on $\widetilde{\mathbf{A}}$ and we define the preconditioner as
\begin{equation}
\label{eq:prec_geo}
\widehat{\mathbf{A}}^{\vect{G}}:= \mathbf{D}^{\tfrac12} \widetilde{\mathbf{A}} \mathbf{D}^{\tfrac12} 
\end{equation}
 where $\displaystyle{
[\mathbf{D}]_{i,i}:=\frac{[\mathbf{A}]_{i,i}}{[\widetilde{\mathbf{A}}]_{i,i}}}  
 $ for $i=1,\dots,N_{dof}$.

\subsubsection{The case of non-constant separable coefficients}
\label{sec:sep_coeff}
We briefly discuss a generalization of the preconditioning strategy to the case of non-constant   equation coefficients $\gamma$ and $\nu$. We assume that $\gamma$ and $\nu $ are positive functions defined over $\Omega \times [0,T]$ and that they are separable in space and in time, i.e. we can write
\[ \gamma(\vect{x},t) = \gamma_s(\vect{x}) \gamma_t(t), \qquad \nu(\vect{x},t) = \nu_s(\vect{x}) \nu_t(t), \]
with $\gamma_s, \gamma_t, \nu_s$ and $\nu_t$ positive functions. 

Now, the first equation of \eqref{eq:problem} can be written as 
\[ \gamma_s \partial_t u - \nabla \cdot \left( \frac{\nu_t}{\gamma_t} \nu_s \nabla u \right) = \frac{f}{\gamma_t}. \]

  We discretize this equation    as described in Section \ref{sec:problem} and we   generalize the definition of the linear system \eqref{eq:syst_mat}  with  
\[ \mathbf{A}  :=   \mathbf{W}_t \otimes   \underline{\mathbf{M}}_s + \underline{\mathbf{M}}_t\otimes  \underline{\mathbf{K}}_s, \]
where $\mathbf{W}_t$ is defined as in \eqref{eq:time_mat}, while
for $i,j=1,\dots,n_t$
\[ [\underline{\mathbf{M}}_t]_{i,j} := \int_{0}^T\, \frac{\nu_t(t)}{\gamma_t(t)} b_{i, p_t}(t)\,  b_{j, p_t}(t)  \, \dt  \]
and for $i,j=1,\dots,N_s$
\[ [\underline{\mathbf{M}}_s]_{i,j} := \int_{\Omega} \gamma_s(\vect{x}) B_{i, \vect{p}_s}(\vect{x}) \  B_{j, \vect{p}_s}(\vect{x}) \ \d\Omega \ \text{ and } \ [ \underline{\mathbf{K}}_s]_{i,j}  :=  \int_{\Omega} \nu_s(\vect{x}) \nabla  B_{i,\vect{p}_s}(\vect{x})\cdot \nabla  B_{j, \vect{p}_s}(\vect{x}) \ \d\Omega. \]  
 
Then, the preconditioner that we propose is defined as in \eqref{eq:prec_geo}
\[ \widehat{\mathbf{A}}^{\vect{G}}   := \mathbf{D}^{\tfrac12}\widetilde{\mathbf{A}}\mathbf{D}^{\tfrac12}, \]
but here  we generalize \eqref{eq:over_A} with 
\[ \widetilde{\mathbf{A}}:= \mathbf{W}_t \otimes  \  \breve{{\mathbf{M}}}_s + \underline{\mathbf{M}}_t\otimes  \breve{ {\mathbf{K}}}_s, \]
where the matrices $\breve{{\mathbf{K}}}_s$ and $\breve{{\mathbf{M}}}_s$ are obtained by using an approximation technique analogous to the one described previously in this section, with $\gamma_s$ and $\nu_s$   included in the coefficient matrix $\mathfrak{C}$. 
  The preconditioner $\widetilde{\mathbf{A}}$ can still be applied as described in Section \ref{sec:prec_appl}. Note that, for this purpose, it is crucial that $\mathbf{W}_t$ does not incorporate any time-dependent coefficient, since this would invalidate \eqref{eq:integral_W}.

   \subsection{Computational cost and memory requirement} 
   \label{sec:cost_memory} 
The   matrix \eqref{eq:syst_mat}  is neither positive definite  nor symmetric and we choose GMRES as linear solver for the system \eqref{eq:sys_solve}.
In GMRES, the orthogonalization of the basis of the Krylov subspace makes the computational cost nonlinear with respect to the number of iterations.
However, as long as this number is not too high, at each iteration the two dominant costs are the application of the preconditioning strategy and the computation of the residual. 

 We assume, for simplicity that for $l=1,\dots, d$ the matrices $\widehat{\mathbf{K}}_l$,
$\widehat{\mathbf{M}}_l$ and $\widetilde{\mathbf{K}}_l$,
$\widetilde{\mathbf{M}}_l$ have dimensions $n_s\times n_s$ and that the matrices $\mathbf{W}_t$,
$\mathbf{M}_t$ and $\widetilde{\mathbf{W}}_t$,
$\widetilde{\mathbf{M}}_t$ have dimensions $n_t \times n_t$. Thus  the  total number of degrees-of-freedom is $N_{dof}=N_sn_t=n_s^dn_t$.

 The setup  of $\widehat{\mathbf{A}}$ and  $\widehat{\mathbf{A}}^{\vect{G}}$ includes the operations performed in Step 1 of Algorithm \ref{al:direct_P}, i.e. $d$ spatial eigendecompositions, that have a total cost of $O(dn_s^3)$ FLOPs, and the factorization of the time matrices. The computational cost of  the latter, that is  the sum of the cost of the eigendecomposition \eqref{eq:factorization_first_block} and of the cost to compute  the solution $\mathbf{v}$ of the linear system \eqref{eq:syst_v},  yields a cost of $O(n_t^3)$ FLOPs.  Then, the total   cost of the spatial and time factorizations  is   $
  O(dn_s^3 + n_t^3)$ FLOPs. Note that, if $n_t = O(n_s)$, this cost is optimal for $d=2$	and negligible for $d=3$.
  The setup cost of $\widehat{\mathbf{A}}^{\vect{G}}$ includes also the the construction of the diagonal matrix $\mathbf{D}$, that has a negligible cost, and the computation of the $2d$ approximations   $\varphi_1,\dots,\varphi_{d}$ and $\Phi_1,\dots,\Phi_{d}$ in \eqref{eq:separation},  whose cost is  negligible  too, as mentioned in Section \ref{sec:prec_geo}.
   We remark that the setup of the preconditioners has to be performed only once, since the matrices involved do not change during the iterative procedure.
  
The application of the preconditioner is performed by Steps 2-4 of Algorithm \ref{al:direct_P}.
 Exploiting \eqref{eq:kron_vec_multi},  Step 2 and Step 4 costs $4(dn_s^{d+1}n_t + n_t^2n_s^d) = 4 N_{dof}(dn_s+n_t)$ FLOPs.  
The use of the  block LU decomposition \eqref{eq:block_LU} makes the cost for Step 3 equal to  $O(N_{dof})$  FLOPs.  

In conclusion, the total cost of Algorithm \ref{al:direct_P} is $4 N_{dof}(dn_s+n_t) +O(N_{dof})$ FLOPs. The non-optimal dominant cost of Step 2 and Step 4 is determined by the dense matrix-matrix products.
 However, these operations are usually implemented on modern computers in a very efficient way.   For this reason, in our numerical tests, the overall serial computational time grows almost as $O(N_{dof})$, see Figure \ref{fig:setup_oneapp} in Section \ref{sec:numerical-tests}.
  
The other dominant computational cost in a GMRES iteration is the cost of the residual computation, that is the multiplication of the matrix $\mathbf{A}$  with a vector. This multiplication is done by exploiting the special structure \eqref{eq:syst_mat}, that allows a matrix-free approach and the use of formula \eqref{eq:kron_vec_multi}. Note in particular that we do not need to compute and to store the whole matrix $\mathbf{A}$, but only its time and spatial factors. 
Since the time matrices  $\mathbf{M}_t$ and $\mathbf{W}_t$ are banded with a band of width $2p_t+1$ and the spatial matrices $\mathbf{K}_s$ and $\mathbf{M}_s$ have roughly $N_s (2p_s+1)^d$ nonzero entries, we have that the computational cost of a single matrix-vector product is $6 N_{dof} [(2p_s+1)^d+2p_t+1] \approx 6 N_{dof} (2p+1)^d = O( N_{dof} p^d )$ FLOPs, if we assume $p=p_s\approx p_t$. The numerical experiments reported in Table \ref{tab:rev_quarter_time} of Section \ref{sec:numerical-tests} show that the dominant cost in the iterative solver is represented by the residual computation. This is a typical behaviour of the FD-based preconditioning strategies, see    \cite{Montardini2018space,Montardini2018,Sangalli2016}.

  We now investigate the memory consumption. For the preconditioner we have to store  the eigenvector spatial matrices $\mathbf{U}_1, \dots, \mathbf{U}_d$, the  time matrix $\mathbf{U}_t$  
and the block-arrowhead matrix \eqref{eq:syst-arr}. The memory required is roughly 
\[n_t^2+dn_s^2+2N_{dof}.\]
 For the system matrix, we have to store the time factors $\mathbf{M}_t$ and $\mathbf{W}_t$ and the spatial factors $\mathbf{M}_s$ and $\mathbf{K}_s$. Thus the memory required is roughly
 \[
 2(2p_t+1)n_t+2(2p_s+1)^d N_s \approx 4 p_t n_t + 2^{d+1} p_s^d N_s .
 \] 
Analogously to  the least-squares case of \cite{Montardini2018space}, we conclude that, in terms of memory requirement, our approach is very attractive w.r.t. other approaches, e.g. the ones obtained by discretizing in space and in time separately. 
For example if we assume $d=3$, $p_t\approx p_s=p$ and $n_t^2\leq Cp^3N_s$, then the total memory consumption is $O(p^3N_s+N_{dof})$, that is equal to the sum of the  memory needed to store the Galerkin matrices associated to spatial variables and the   memory needed  to store the solution of the problem.

 We remark that we could avoid storing the factors of $\mathbf{A}$ by using the matrix-free approach of \cite{Sangalli2018}. The memory and the computational cost of the iterative solver would significantly improve,  both for the setup and the matrix-vector multiplications. However, we do not pursue this strategy, as it is beyond the scope of this paper.
 
 \begin{rmk}
For a better computational efficiency,
we use a real-arithmetic version of Algorithm \ref{al:direct_P}: { we replace }  $ \widetilde{\mathbf{\Lambda}_t}$ in {  \eqref{eq:time_eig_2}}
by a block diagonal matrix where  each pair
of generalized eigenvalues $i\lambda_j$ and $-i\lambda_j$ is replaced
by a  diagonal block   
 \[
 \begin{bmatrix}
 0 & \lambda_j\\
 -\lambda_j & 0
 \end{bmatrix}
 \]
  and   we set 
 \[\mathbf{H}_j:= \begin{bmatrix}
 { \nu}\mathbf{\Lambda}_s & {  \gamma}\lambda_j\mathbb{I}_{n_s}\\
 -{  \gamma}\lambda_j\mathbb{I}_{n_s} &  { \nu}\mathbf{\Lambda}_s
 \end{bmatrix}\quad
   \text{ and }\quad \mathbf{B}_j:={ \gamma}\left[ [\mathbf{g}]_{2(j-1)+1}\mathbb{I}_{N_s},  \quad [\mathbf{g}]_{2(j-1)+2}\mathbb{I}_{N_s}\right]^T.\] 
 Note that the computational cost of Step 3 in Algorithm \ref{al:direct_P} does not change, as we have   
 \[
 \mathbf{H}_j^{-1}:=\begin{bmatrix}
 {  \tfrac1\nu}\mathbf{\Lambda}_s^{-1}-{ \frac{\gamma^2}{\nu^2}}\lambda_j^2 \mathbf{\Lambda}_s^{-1}\mathbf{R}_j^{-1}\mathbf{\Lambda}_s^{-1} & -{  \frac{\gamma }{\nu }}\lambda_j\mathbf{\Lambda}_s^{-1}\mathbf{R}_j^{-1}\\[5pt]
{ \frac{\gamma}{\nu}}\lambda_j\mathbf{R}_j^{-1}\mathbf{\Lambda}_s^{-1} &  \mathbf{R}_j^{-1}  
 \end{bmatrix}.
 \]
 where $\mathbf{R}_j:={  \nu}\mathbf{\Lambda}_s + { \frac{\gamma^2}{\nu}}\lambda_j^2\mathbf{\Lambda}_s^{-1}. $
 \end{rmk}
  

\section{Numerical Results}
\label{sec:numerical-tests}
{
In this section we first present the numerical experiments that assess the convergence behavior  of the Galerkin approximation and then we analyze the performance of the preconditioners. We also present a comparison with the 
  the  least-squares  solver of \cite{Montardini2018space}.

We consider only sequential executions and we force the use of  a single computational thread in a Intel Core i7-5820K processor, running at 3.30 GHz and with 64 GB of RAM.
 
 The tests are performed with Matlab R2015a and GeoPDEs toolbox \cite{Vazquez2016}. 
We use  the \texttt{eig} Matlab function to compute the  generalized eigendecompositions present in  Step 1 of Algorithm \ref{al:direct_P}
, while  Tensorlab toolbox \cite{Sorber2014} is employed to perform the  multiplications with Kronecker matrices occurring in Step 2 and Step 4.  The  solution of the linear system \eqref{eq:syst_v}  is performed by Matlab  direct solver (backslash operator ``$\backslash$'').
The linear system is solved by GMRES, with tolerance   equal to  $10^{-8}$  and with   the null vector as initial guess in all tests. We remark that GMRES computes and stores a full orthonormal basis for the Krylov space, and this might be unfeasible if the number of iterations is too large. This issue could be addressed by switching to a different solver for nonsymmetric systems, like e.g. BiCGStab, or using the restarted version of GMRES.

According to Remark \ref{rem:on-the-error-bound}, we use  the same
mesh-size in space and in time $h_s=h_t=:h$,  and use
splines of maximal continuity  and  same degree in
space and in time  $p_t=p_s=:p$. For the sake of simplicity, we
also consider uniform knot vectors, and  denote the number of
elements in each parametric direction  by $n_{el}:=\frac{1}{h}$.

In out tables, the symbol $``\ast"$   denotes  that the construction of the matrix  factors of $\mathbf{A}$ (see \eqref{eq:syst_mat}) goes out of memory,
  while the symbol  $``\ast \ast"$ indicates that the dimension of the
  Krylov subspace is too high and there   is   not  enough memory to store all the GMRES iterates. 
 We remark that  in all the tables the total solving time of the iterative strategies includes also the setup time of the considered preconditioner.
 
\begin{figure}
 \centering
 \subfloat[][Rotated quarter of annulus.\label{fig:rev-quarter}]
   {\includegraphics[trim=5cm 0cm 5cm 0cm, clip=true, scale=0.25]{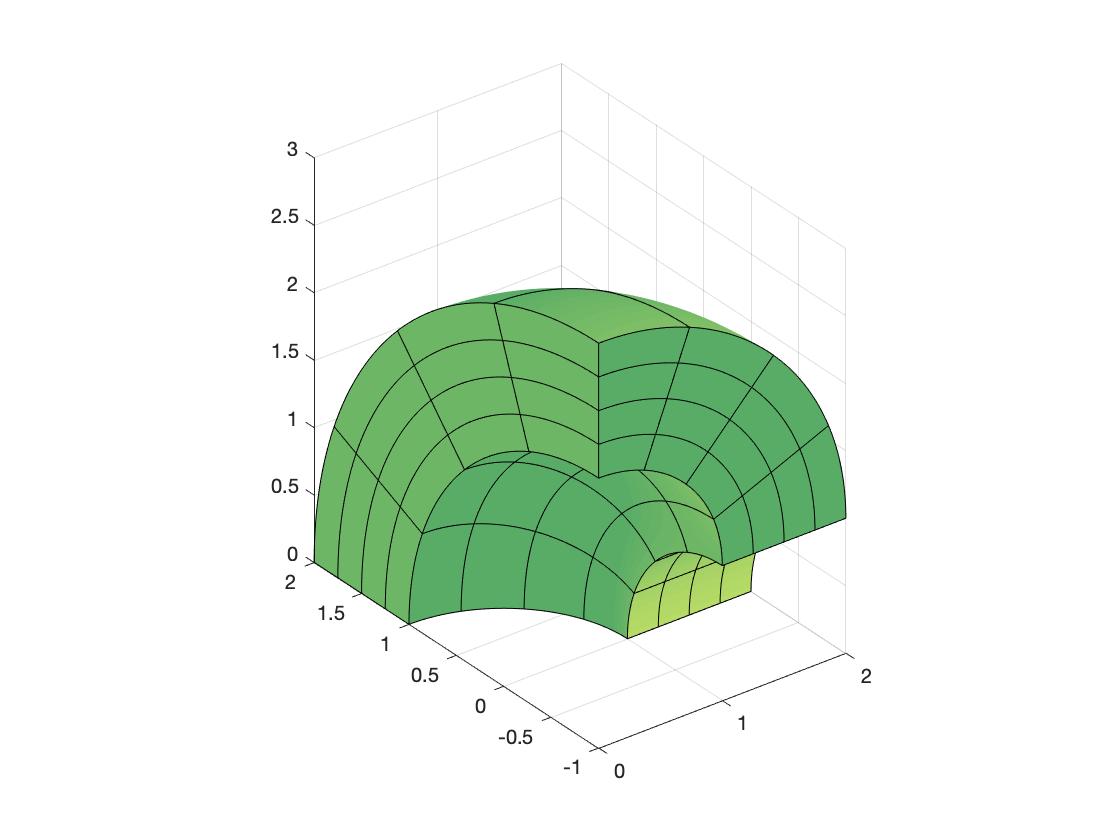}}\\ 
  \subfloat[][Hollow torus.\label{fig:hollow-torus}]
   {\includegraphics[trim=5cm 0cm 5cm 0cm, clip=true, scale=0.20]{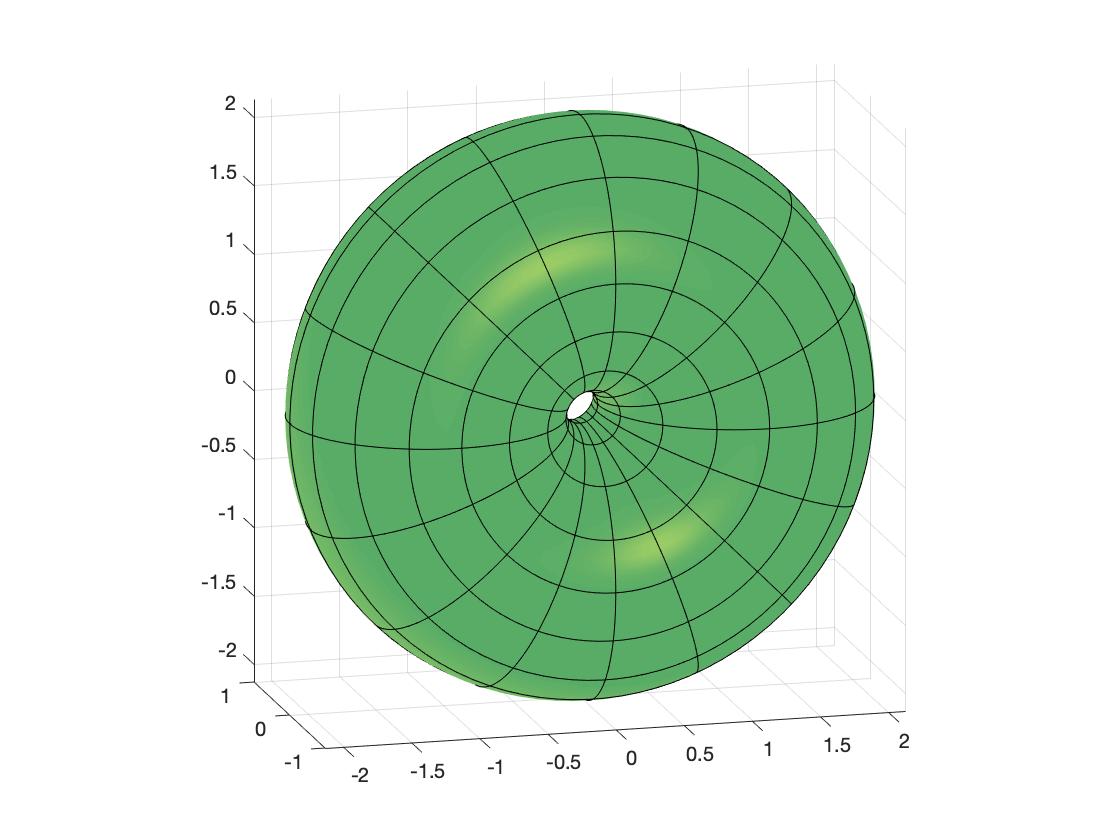}}  \quad
   \subfloat[][Section of the hollow torus.\label{fig:section_hollow}] 
   {\includegraphics[trim=5cm 0cm 5cm 0cm, clip=true, scale=0.20]{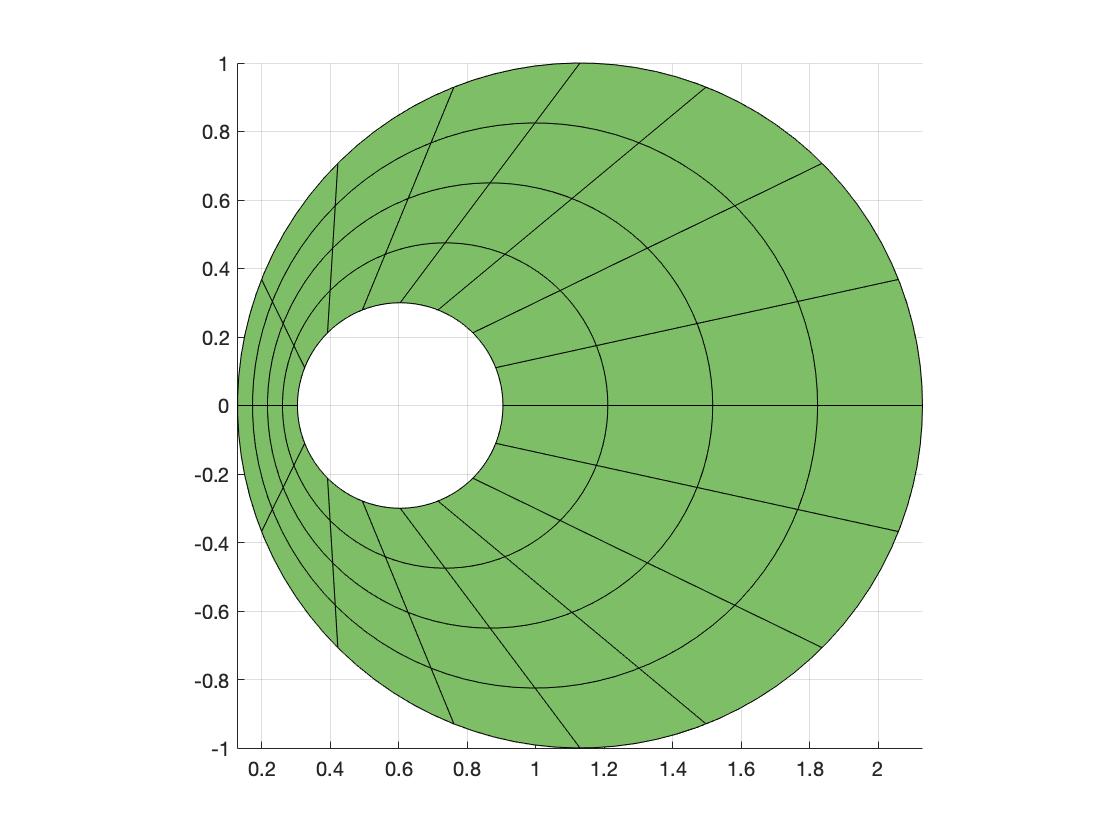}}\\
 \caption{Computational domains.  }
 \label{fig:Geometries}
\end{figure}

 \subsection{Orders of convergence}
 \label{sec:orders_conv}
We consider as spatial computational domain $\Omega$ a rotated quarter of annulus, represented in Figure \ref{fig:rev-quarter}: we rotate by $\tfrac\pi2$ a quarter of annulus with center in the origin, internal radius 1 and external radius 2 
  along the axis $\{(x_1,-1,0) \ | \ x_1\in\mathbb{R}\}$.  Dirichlet  and initial boundary conditions are set such that $ u(\vect{x},t)= -(x_1^2+x_2^2-1)(x_1^2+x_2^2-4)x_1x_2^2\sin(t)\sin(x_3)$ 
  is the exact solution with constants $\nu=\gamma=1$.

In Figure \ref{fig:err_X} we represent the relative errors in
$L^2(0,T;H^1_0(\Omega))\cap H^1(0,T;L^2(\Omega))$  norm, an easily
computable upper bound of $\|\cdot \|_{\mathcal{X}_h}$,  for polynomial degrees $p=1, 2, 3, 4, 5$. 
The rates of convergence are optimal, i.e. of order $O(h^{p})$, consistent with the a-priori estimate \eqref{eq:a-riori-error-bound}.
Even if this case is not covered by theoretical results, we also compute the relative errors in  $L^2( 0,T;L^2(\Omega))$ norm: the orders of convergence are  still optimal, that is of order $O(h^{p+1})$, as Figure \ref{fig:err_L2} shows.

\begin{figure}
 \centering
  \subfloat[][$L^2(0,T;H^1_0(\Omega))\cap H^1(0,T;L^2(\Omega))$ norm relative errors.\label{fig:err_X}]
   {\includegraphics[scale=0.25]{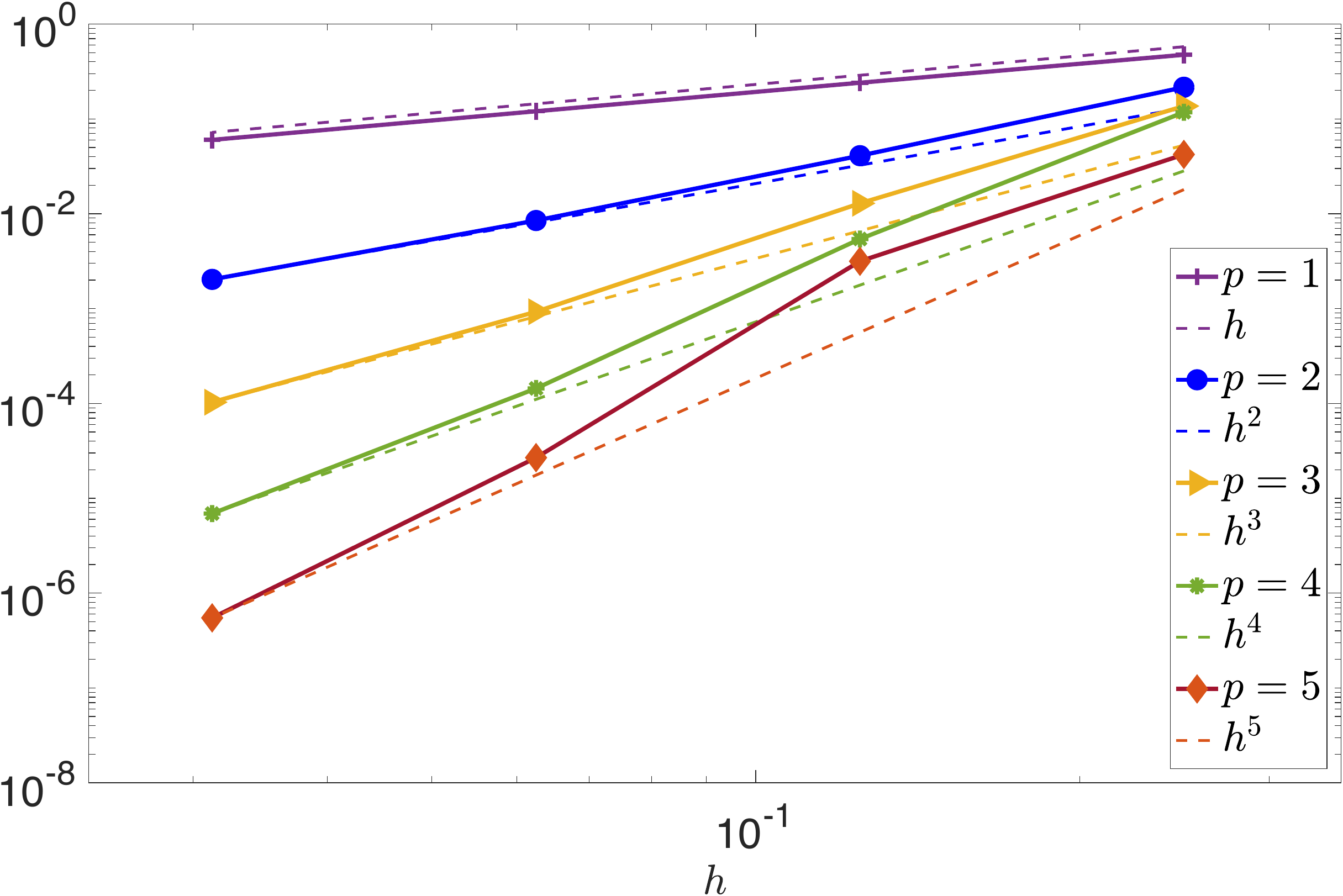}}  \quad  \subfloat[][$L^2( 0,T;L^2(\Omega))$ norm relative errors.\label{fig:err_L2}]
   {\includegraphics[scale=0.25]{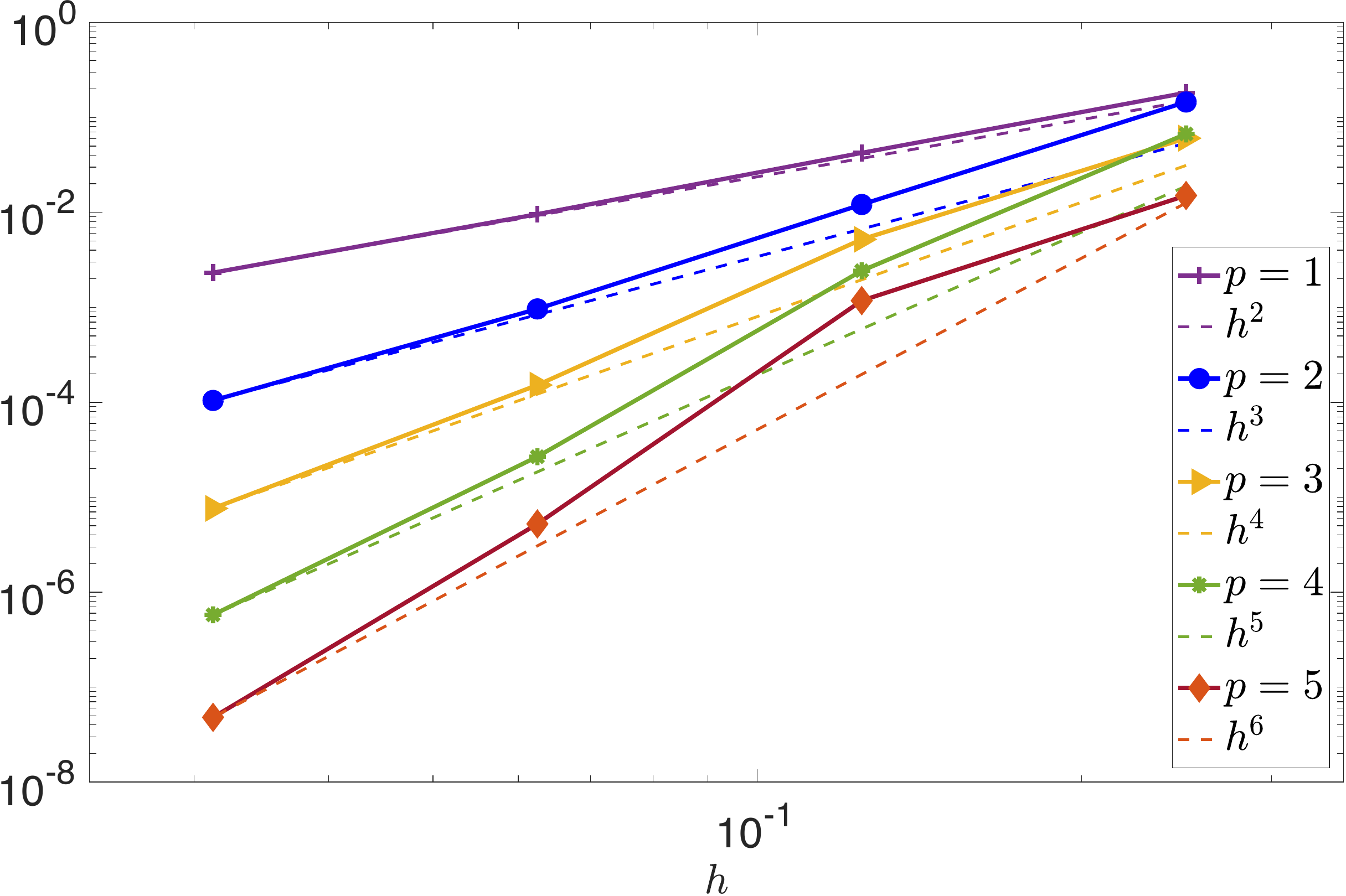}}
 \caption{ Relative errors.}
 \label{fig:errors}
\end{figure}

 \subsection{Performance of the preconditioner: rotated quarter of annulus}
We consider again as spatial computational domain $\Omega$ the rotated quarter of annulus  of Figure \ref{fig:rev-quarter}   and   the same exact solution, initial and boundary data as in   Section \ref{sec:orders_conv}. We analyze the performance of both $\widehat{\mathbf{A}}$ and $\widehat{\mathbf{A}}^{\vect{G}}$. 
The maximum dimension of the Krylov subspace is set equal to 100 for both the preconditioners  up to $n_{el}=64$. We are able to reach convergence and to perform the tests with  $\widehat{\mathbf{A}}^{\vect{G}}$,  $n_{el}=128$  and $p=1,2,3$ by  setting the maximum Krylov subspace dimension equal to 25.  
In Table \ref{tab:rev_quarter} we report   the number of iterations and the total solving time  of GMRES preconditioned with   $\widehat{\mathbf{A}}$  (upper table) and $\widehat{\mathbf{A}}^{\vect{G}}$ (middle table). 
The non-trivial geometry  clearly affects the performance of
$\widehat{\mathbf{A}}$, but,  when we include some information on the
parametrization by using  $\widehat{\mathbf{A}}^{\vect{G}}$, the number of iterations is more than halved and it is stable w.r.t. $p$ and $n_{el}$. Moreover, the computational times are one order of magnitude lower for  the highest degrees and $n_{el}$.
  In the lower table of Table \ref{tab:rev_quarter} we report the
  results of  \cite[Section 5, central table of Table 2]{Montardini2018space} obtained by
  solving the same problem with  the least-squares formulation and the
  related preconditioning strategy.    In this case the iterative
  solver  is the preconditioned conjugate gradient method, the
  tolerance is $10^{-8}$ and the initial guess is the null vector.
  The number of iterations is more than doubled and the computational
  times are three times higher   than the number of iterations and
  computational times of $\widehat{\mathbf{A}}^{\vect{G}}$, in the
  present setting.
 {\renewcommand\arraystretch{1.4} 
\begin{table}
\begin{center}
\begin{tabular}{|r|c|c|c|c|c|}
\hline
& \multicolumn{5}{|c|}{ \  $\widehat{\mathbf{A}}$\quad Iterations /  Time } \\
\hline
$n_{el}$ & $p=1$ &$p=2$ & $p=3$ & $p=4$ & $p=5$   \\
\hline
8 & 34 / \z\z0.20   &  37 / \z\z0.21 &  42  / \z\z\z0.42  &   46 / \z\z\z0.63    &     50 / \z\z\z1.13     \\
\hline
16 & 43 / \z\z1.15  &  46  / \z\z1.65   &  50 /   \z\z\z3.42      &   54 / \z\z\z5.80  &   57 /  \z\z11.87    \\
\hline
32 & 50 / \z22.75  & 53 / \z31.10 &   57 / \z\z54.02  &   61 /   \z\z96.06   &  64  / \z184.84   \\
\hline
64 & 57 / 586.73  & 60  /  764.26 &  67 / 1254.81   &  67 / 1858.55      &  71 /  3188.51 \\
\hline
128 &   $\ast\ast$  &    $\ast\ast$  &    $\ast\ast$   &    $\ast$    &  $\ast$   \\
\hline
\end{tabular} 
\vskip 2mm
\begin{tabular}{|r|c|c|c|c|c|}
\hline
& \multicolumn{5}{|c|}{ \   $\widehat{\mathbf{A}}^{\vect{G}}$\quad Iterations / Time } \\
\hline
$n_{el}$ & $p=1$ & $p=2$ & $p=3$ & $p=4$ & $p=5$  \\
\hline
8 &  11 / \z\z\z0.06 & 12 / \z\z\z0.09   &  12 / \z\z\z0.11    &   13 / \z\z0.18   &    14 /  \z\z0.29  \\
\hline
16 & 13  / \z\z\z0.26  & 14  / \z\z\z0.52 &   14 / \z\z\z1.18  &  14 /  \z\z1.44     &   15 /   \z\z3.85     \\
\hline
32 & 15 / \z\z\z4.73  & 15  / \z\z\z6.76     &  15 / \z\z12.67  &    15 /  \z21.47  &  16 /  \z40.54  \\
\hline
64 & 16 / \z107.24 & 16 /  \z135.74     &  18 / \z249.27  &    16 /  370.31  & 17 / 695.44 \\
\hline
128 &  17 /  2623.57 &    17  /   3105.76        &  17 / 5614.10  &   $\ast$   & $\ast$   \\
\hline
\end{tabular}
\vskip 2mm
\begin{tabular}{|r|c|c|c|c|}
\hline
& \multicolumn{4}{|c|}{Least-squares \quad Iterations / Time } \\
\hline
$n_{el}$ & $p_t=2$ & $p_t=3$ & $p_t=4$ & $p_t=5$  \\
\hline
8 &   24 / \z\z\z0.09 & 24  / \z\z\z\z0.13  &  26  /  \z\z\z0.37     &  26 /  \z\z\z0.60     \\
\hline
16 &   35 / \z\z\z0.77  & 34 / \z\z\z\z1.96    &  33 / \z\z\z4.62   &      33 /  \z\z\z9.35    \\
\hline
32 &   42 / \z\z17.03    & 41 / \z\z\z39.57  &  40 / \z\z82.35  & 41 / \z161.73  \\
\hline
64 & 46 / \z333.20 & 44 / \z\z716.03 & 49 / 1577.55 & 53 / 3384.08  \\
\hline
128 & 48 / 6767.08   & 50 / 14814.09  & $\ast$ & $\ast$\\
\hline
\end{tabular}
\end{center}
\caption{Revolved quarter domain. Performance of $\widehat{\mathbf{A}}$, $ \widehat{\mathbf{A}}^{\vect{G}}$   and the least-squares solver. }
\label{tab:rev_quarter}
\end{table}}

 Finally, we analyze with more details the performance of $\widehat{\mathbf{A}}^{\vect{G}}$.
First, we consider the percentage of time spent in the   application of  $\widehat{\mathbf{A}}^{\vect{G}}$  in one GMRES iteration. The results, reported in Table \ref{tab:rev_quarter_time}, clearly show that the dominant cost consists of the matrix-vector multiplications, while  the application of the preconditioner takes a small percentage of the total computational time, for example less than $10\%$ for polynomial degree 5 and $n_{el}=32$ or $n_{el}=64$.
 In Figure \ref{fig:setup_oneapp} we report     the setup time and
 the single application time of $\widehat{\mathbf{A}}^{\vect{G}}$
 w.r.t. the number of degrees of freedom.  As expected, the setup time
 is proportional to $O(N_{dof})$.   What is more interesting is that
 the application time grows slower than $O(N_{dof}^{5/4})$, i.e.  the
 FLOPS counting, and it is almost proportional to $O(N_{dof})$: this may
 be explained by the fact that the memory access is the dominant cost
due to the high-efficiency of CPU operations, in our case implemented
in Matlab Tensorlab \cite{Sorber2014}.

 {\renewcommand\arraystretch{1.4} 
\begin{table}
\begin{center}
\begin{tabular}{|r|c|c|c|c|c|}
\hline
$n_{el}$ & $p=1$ &$p=2$ & $p=3$ & $p=4$ & $p=5$   \\
\hline
8 &  73.02 \% &  79.24 \%  &  66.62 \%  &   46.94 \%   &   33.73  \%   \\
\hline
16 &  68.10  \% &  46.13 \% &  30.06 \%   &   17.63 \%  &  11.27  \%   \\
\hline
32 & 53.09 \% & 33.34 \% &   20.44 \% &  13.06 \%   & \z8.19 \%  \\
\hline
64 & 54.71 \% & 32.46 \% &   20.20 \%  &   12.52  \%   & \z7.31 \%  \\
\hline
128 & 54.12 \%  &  33.53 \%  &    18.89 \%  &    $\ast$    &  $\ast$   \\
\hline
\end{tabular}
\caption{Percentage of computing time of $\widehat{\mathbf{A}}^{\vect{G}}$ in one GMRES iteration for the rotated quarter domain.}
\label{tab:rev_quarter_time}
\end{center}
\end{table}}

\begin{figure}
 \centering
    {\includegraphics[scale=0.53]{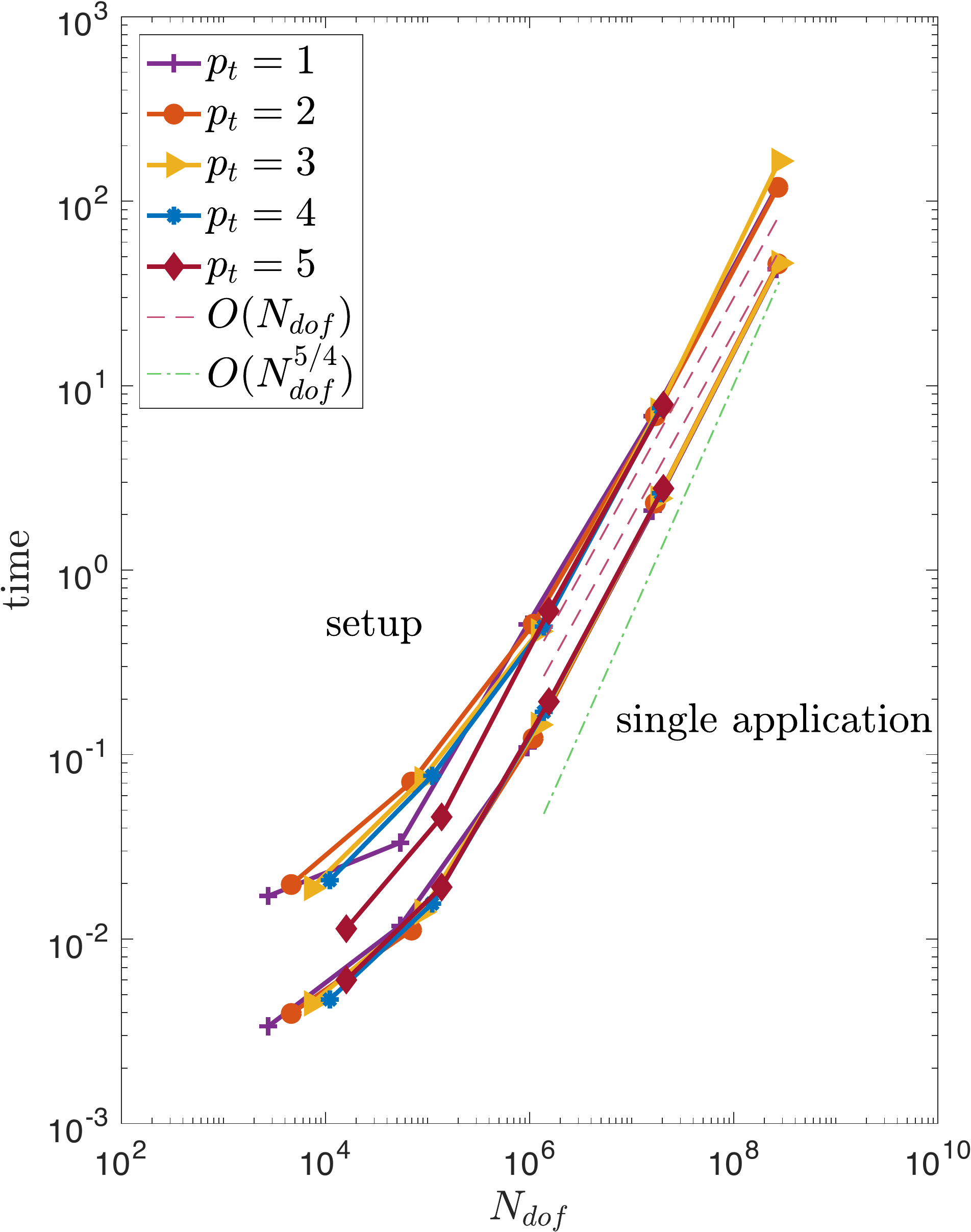}}
 \caption{Setup time and single application time of $\widehat{\mathbf{A}}^{\vect{G}}$ in the rotated quarter domain.}
 \label{fig:setup_oneapp}
\end{figure}

\subsection{Performance of the preconditioner: hollow torus}
We consider a torus with a hole (Figure \ref{fig:hollow-torus}) that is obtained by revolving an eccentric annulus (Figure \ref{fig:section_hollow}) along the $x_2$-axis. 
 For this problem we consider $\gamma =1$ and a separable in spatial and time variables, non-constant diffusion coefficient $\nu(\vect{x},t)=\nu_s(\vect{x})\nu_t(t)$. Precisely, we choose 
\[
\nu_s(\vect{x}):= 1+ \frac{99}{2}\left[1+ \frac{1}{{\left(1+\frac{x_1^2}{x_3^2}\right)}^{\tfrac12}}\right] \text{ and } \nu_t(t):= 1+ 50 \left[1+\cos\left(\frac{t}{2\pi}\right)\right] \ .
\] We remark that we are in the  setting described in Section \ref{sec:sep_coeff}.   The initial data and right-hand side are defined such that  \[
 u(\vect{x},t):= \sin(\pi x_1) \sin(\pi x_2) \sin(\pi x_3) \sin(\pi t)\] is the exact solution.
 In this case, we replace $\nu$ in \eqref{eq:prec_definition} with  its integral mean $  \tfrac1{T|\Omega|}\int_0^T\int_{\Omega}\nu(\vect{x}, t)\ \d\Omega\ \dt $.   
In Table \ref{tab:hollow_torus} we compare the performance of $\widehat{\mathbf{A}}$ (upper table) and  $\widehat{\mathbf{A}}^{\vect{G}}$ (lower table): the inclusion of the information about the geometry parametrization and of the variable  coefficient  significantly reduces   the number of iterations and the computational  times.

  {\renewcommand\arraystretch{1.4} 
\begin{table}
\begin{center} {
\begin{tabular}{|r|c|c|c|c|c|}
\hline \normalsize
& \multicolumn{5}{|c|}{ \ $\widehat{\mathbf{A}}$\quad Iterations /  Time } \\
\hline
$n_{el}$ & $p=1$ &$p=2$ & $p=3$ & $p=4$ & $p=5$   \\
\hline
8 & \z32  / \z\z\z0.49    &  \z70 /  \z\z\z0.79   &   101 /  \z\z\z2.02  &   128 /  \z\z\z\z5.83   &    156  / \z\z\z14.48      \\
\hline
16 & \z98 / \z\z\z5.83   & 121  /  \z\z10.54    & 149 /   \z\z26.13      &   167 / \z\z\z57.27   &  177  /     \z\z128.68  \\
\hline
32 & 143  / \z122.28   &  165  /  \z236.47 &    177  /   \z400.79 &        193 / \z\z746.28 &  197 / \z1230.60    \\
\hline
64 & 165  / 3657.33   &   168 / 4733.98  &     175  /  6596.99   &     179     / 15894.01 &    184 / 20215.23 \\
\hline
\end{tabular} }
\vskip 2mm  {
\begin{tabular}{|r|c|c|c|c|c|}
\hline 
& \multicolumn{5}{|c|}{ \  $\widehat{\mathbf{A}}^{\vect{G}}$\quad Iterations / Time } \\
\hline
$n_{el}$ & $p=1$ & $p=2$ & $p=3$ & $p=4$ & $p=5$  \\
\hline
8 & 14 /  \z\z0.30 & 15 / \z\z0.50  &  19 / \z\z0.71    &   20 /\z\z\z1.11   &    23 / \z\z\z1.98 \\
\hline
16 &  18 / \z\z0.87 &  19 / \z\z1.66 &   21 / \z\z2.79 &  23 / \z\z\z5.77    &  25 / \z\z14.12     \\
\hline
32 & 22 / \z\z8.88  &  24 / \z16.08    &  25 / \z29.66  &  26 / \z\z61.22  &    27 / \z114.93 \\
\hline
64 &  26 / 207.70   &   27 /  303.33 &     28  /  495.29   &          29 / 1118.44 &  30  / 1923.20 \\
\hline  
\end{tabular}}
\end{center}
\caption{Hollow torus domain. Performance of $\widehat{\mathbf{A}}  $  and $ \widehat{\mathbf{A}}^{\vect{G}}$.}
\label{tab:hollow_torus}
\end{table}}
 }
\section{Conclusions}
 In this work we proposed a preconditioner suited for a space-time
 Galerkin  isogeometric discretization of the heat equation. Our
 preconditioner $\widehat{\mathbf{A}}$ is represented by a suitable
 sum of Kronecker products of matrices, that makes the computational
 cost of its construction (setup) and application, as well as the storage cost, very appealing.
In particular the application of the preconditioner, inspired by
the  FD technique,  exploits an ad-hoc factorization
of the time matrices. The preconditioner  cost seen in  numerical  tests,  for a serial
single core execution, is almost equal to $O(N_{dof})$ and does not
depend on the polynomial degree. 

 
At the same time, the storage cost is roughly the same that we would have by discretizing separately in space and in
time, if we assume $n_t\leq Cp^dN_s$. Indeed, in this case the memory
used for the whole iterative solver is $O(p^dN_s+N_{dof}).$


In this paper, we have restricted ourselves to the case of a fixed domain and of constant (or separable) coefficients.
However, the proposed approach can be extended to the case where the domain changes over time and/or the coefficients of the equation are not separable.
Clearly, in these cases the matrix $\mathbf{A}$ is no longer the sum of Kronecker products as in \eqref{eq:syst_mat}, and its storage is likely unfeasible in practical problems.
A possible way to circumvent this issue is to switch to a matrix-free approach \cite{Sangalli2018}, where the matrix is not stored and is available only to compute matrix-vector products.
To build the preconditioner, the integral kernels that appear in the matrix entries should be replaced by separable approximations.
This can be done using the same technique described in Section \ref{sec:prec_geo}, at the (optimal) cost of $O(N_{dof})$ flops.
A similar approach can be used if we consider a nonlinear problem, where a linear system of the form \eqref{eq:sys_solve} has to be solved at each step of a nonlinear iteration.
Note that in this case the preconditioner has to be build from scratch every time, as the matrix $\mathbf{A}$ changes at every iteration. 
This, however, is not an issue, since  as discussed in Section \ref{sec:cost_memory} the setup cost for the preconditioner is optimal (or even negligible) and independent of $p$. 

As a final comment, we mention that our method has a strong potential for parallelization, and this will be an interesting future direction of study.

\section*{Acknowledgments}

The authors were partially supported by the European Research Council
through the FP7 Ideas Consolidator Grant HIGEOM n.616563. The
authors are members of the Gruppo Nazionale Calcolo
Scientifico-Istituto Nazionale di Alta Matematica (GNCS-INDAM)  and
the second  author  was partially supported by  INDAM-GNCS
``Finanziamento Giovani Ricercatori 2019-20" for the project
``Efficiente risoluzione dell'equazione di Navier-Stokes in ambito
isogeometrico". These supports are gratefully acknowledged.

\bibliography{biblio_space_time}
\end{document}